\documentclass[11pt]{article}
\usepackage{amsmath,amssymb,amsthm,mathrsfs}
\usepackage[all]{xy}
\usepackage[dvips]{graphicx}

\usepackage{color}

\newtheorem{Def}{Definition}[section]
\newtheorem{Thm}[Def]{Theorem}
\newtheorem{Prop}[Def]{Proposition}

\newtheorem{Lem}[Def]{Lemma}
\newtheorem{Rem}[Def]{Remark}
\newtheorem{Ex}[Def]{Example}

\newtheorem*{keylem}{Key Lemma}

\newcommand{\C}{\mathbb{C}}

\newcommand{\R}{\mathbb{R}}
\newcommand{\Z}{\mathbb{Z}}

\newcommand{\Q}{\mathbb{Q}}
\newcommand{\PP}{\mathbb{P}}
\newcommand{\UU}{\mathrm{U}}
\newcommand{\SU}{\mathop{\mathrm{SU}}\nolimits}
\newcommand{\SL}{\mathop{\mathrm{SL}}\nolimits}

\newcommand{\GL}{\mathop{\mathrm{GL}}\nolimits}
\newcommand{\PGL}{\mathop{\mathrm{PGL}}\nolimits}

\newcommand{\id}{\mathop{\mathrm{id}}\nolimits}
\newcommand{\disc}{\mathop{\mathrm{disc}}\nolimits}
\newcommand{\Ker}{\mathop{\mathrm{Ker}}\nolimits}

\newcommand{\rk}{\mathop{\mathrm{rank}}\nolimits}

\newcommand{\Aut}{\mathop{\mathrm{Aut}}\nolimits}

\newcommand{\HF}{1/2}
\newcommand{\ord}{\operatorname{ord}}
\newcommand{\rank}{\operatorname{rank}}
\newcommand{\sign}{\operatorname{sign}}

\newcommand{\LF}[1]{\langle #1 \rangle}
\newcommand{\QMat}[4]{\begin{bmatrix}#1&#2\\#3&#4\end{bmatrix}}
\newcommand{\invpoly}{s_1^i t_1^{2-i} s_2^j t_2^{2-j} s_3^k t_3^{2-k}  + s_1^{2-i} t_1^{i} s_2^{2-j} t_2^{j} s_3^{2-k} t_3^{k}}

\makeatletter
\@addtoreset{equation}{section}

\makeatother

\begin{document}
\title{Calabi--Yau Threefolds of Type K (I): Classification}
\author{Kenji Hashimoto \ \ \ Atsushi Kanazawa}
\date{}

\maketitle

\begin{abstract}
Any Calabi--Yau threefold $X$ with infinite fundamental group admits an \'{e}tale Galois covering either by an abelian threefold or by the product of a K3 surface and an elliptic curve.   
We call $X$ of type A in the former case and of type K in the latter case. 
In this paper, we provide the full classification of Calabi--Yau threefolds of type K, based on Oguiso and Sakurai's work \cite{OS}.  
Together with a refinement of their result on Calabi--Yau threefolds of type A, we finally complete the classification of Calabi--Yau threefolds with infinite fundamental group. 
\end{abstract}




\section{Introduction}
The present paper is concerned with the Calabi--Yau threefolds with infinite fundamental group.  
Throughout the paper, a Calabi--Yau threefold is a smooth complex projective threefold $X$ with trivial canonical bundle
 and $H^{1}(X, \mathcal{O}_{X})=0$. 
Let $X$ be a Calabi--Yau threefold with infinite fundamental group.  
Then the Bogomolov decomposition theorem \cite{Be1} implies that $X$ admits an \'{e}tale Galois covering 
either by an abelian threefold or by the product of a K3 surface and an elliptic curve.   
We call $X$ of type A in the former case and of type K in the latter case. 
Among many candidates for such coverings, we can always find a unique smallest one, up to isomorphism as a covering \cite{Be2}. 
We call the smallest covering the minimal splitting covering of $X$. 
The main result of this paper is the following: 

\begin{Thm}[Theorem \ref{Main Thm}] \label{THM_intro}
There exist exactly eight Calabi--Yau threefolds of type K, up to deformation equivalence.
The equivalence class is uniquely determined by the Galois group $G$ of the minimal splitting covering. 
Moreover, the Galois group is isomorphic to one of the following combinations of cyclic and dihedral groups
$$
C_{2},\ C_2 \times C_2, \ C_2 \times C_2 \times C_2, \ D_{6}, \ D_{8}, \ D_{10}, \ D_{12}, \ or \ C_2 \times D_8. 
$$ 
\end{Thm}

Most Calabi--Yau threefolds we know have finite fundamental groups: 
for example, complete intersection Calabi--Yau threefolds in toric varieties or homogeneous spaces, and (resolutions of singularities of) finite quotients thereof.   
Calabi--Yau threefolds with infinite fundamental group were only partially explored before the pioneering work of Oguiso and Sakurai \cite{OS}. 
In their paper, they made a list of possible Galois groups 
for type K but it was not settled whether they really occur or not.  
In this paper, we complement their work by providing the full classification (Theorem \ref{THM_intro}) 
and also give an explicit presentation for the deformation classes of the eight Calabi--Yau threefolds of type K. 

The results described in this paper represent the first step in our program which is aimed at more detailed understanding of Calabi--Yau threefolds of type K. 
Calabi--Yau threefolds of type K are relatively simple yet rich enough to display the essential complexities of Calabi--Yau geometries, 
and we expect that they will provide good testing-grounds for general theories and conjectures. 
Indeed, the simplest example, known as the Enriques Calabi--Yau threefold (or the FHSV-model \cite{FHSV}), 
has been one of the most tractable compact Calabi--Yau threefolds  
both in string theory and mathematics (see for example \cite{FHSV,As, KM2,MP}).     
A particularly nice property of Calabi--Yau threefolds of type K is their fibration structure; 
they all admit a K3 fibration, an abelian surface fibration, and an elliptic fibration.   
This rich structure suggests that they play an important role in dualities among various string theories. 
In the forthcoming paper \cite{HK}, we will discuss mirror symmetry of Calabi--Yau threefolds of type K. 

We will also provide the full classification of Calabi--Yau threefolds of type A, again based on Oguiso and Sakurai's work \cite{OS}. 
In contrast to Calabi--Yau threefolds of type K, Calabi--Yau threefolds of type A are classified not by the Galois groups of the minimal splitting coverings, 
but by the {\it minimal totally splitting coverings}, where abelian threefolds 
that 
 cover Calabi--Yau threefolds of type A split into the product of three elliptic curves (Theorem \ref{Refinement}). 
Together with the classification of Calabi--Yau threefolds of type K, we finally complete the full classification of Calabi--Yau threefolds with infinite fundamental group:

\begin{Thm}[Theorem \ref{Classification A and K}]
There exist exactly fourteen deformation classes of
 Calabi--Yau threefolds with infinite fundamental group.
More precisely, six of them are of type A, and eight of them are of type K.
\end{Thm}

It is remarkable that we can study Calabi--Yau threefolds very concretely by simply assuming that their fundamental groups are infinite. 
Recall that a fundamental gap in the classification of algebraic threefolds is the lack of understanding of Calabi--Yau threefolds. 
We hope that our results unveil an interesting class of Calabi--Yau threefolds and shed some light on the further investigation of general compact Calabi--Yau threefolds. 

\subsection*{Structure of Paper}
In Section 2 we recall some basics on lattices and K3 surfaces. 
Lattice theory will be useful when we study finite automorphism groups on K3 surfaces in later sections. 
Section 3 is the main part of this paper. 
It begins with a review of Oguiso and Sakurai's fundamental work \cite{OS}, 
which essentially reduces the study of Calabi--Yau threefolds of type K to that of K3 surfaces equipped with Calabi--Yau actions (Definition \ref{def_cyg_k3}).  
It then provides the full classification of Calabi--Yau threefolds of type K, presenting all the deformation classes. 
Section 4 is devoted to the proof of a key technical result, Lemma \ref{LEM_keylemma} (Key Lemma), which plays a crucial role in Section 3. 
Section 5 addresses some basic properties of Calabi--Yau threefolds of type K. 
Section 6 improves Oguiso and Sakurai's work on Calabi--Yau threefolds of type A. 
It finally completes the classification of Calabi--Yau threefolds with infinite fundamental group. 

\subsection*{Acknowledgement}
It is a pleasure to record our thanks to K. Behrend, J. Bryan, I. Dolgachev, C. Doran, S. Hosono, J. Keum, K. Oguiso, T. Terasoma and S.-T. Yau for inspiring discussions. 
Much of this research was done when the first author was supported by Korea Institute for Advanced Study.
The second author warmly thanks N. Yui for her hospitality at the Fields Institute, where some of this project was carried out.  
The second author is supported by the Center of Mathematical Sciences and Applications at Harvard University. 



\section{Lattices and K3 surfaces}
We begin with a brief summary of the basics of lattices and K3 surfaces.  
This will also serve to set conventions and notations. 
Standard references are \cite{BHPV, Ni2}.

\subsection{Lattices} \label{SUBSECT_lattice}
 A lattice is a free $\Z$-module $L$ of finite rank together with a symmetric bilinear form $\langle *,**\rangle\colon L\times L\to \Z$. 
By an abuse of notation, we denote a lattice simply by $L$. 
 With respect to a choice of basis, the bilinear form is represented by a Gram matrix and the discriminant $\disc(L)$ is the determinant of the Gram matrix.  
We denote by $\mathrm{O}(L)$ the group of
 automorphisms of $L$.
We define $L(n)$ to be the lattice obtained by multiplying the bilinear form $L$ by an integer $n$. 
For $a \in \Q$ we denote by $\langle a \rangle$
 the lattice of rank $1$ generated by $x$ with $x^2:=\langle x, x\rangle=a$. 
A lattice $L$ is called even if $x^2 \in 2\mathbb{Z}$ for all $x \in L$.  
$L$ is non-degenerate if $\disc(L)\neq 0$ and unimodular if $\disc(L)=\pm1$.
If $L$ is a non-degenerate lattice, the signature of $L$ is the pair $(t_{+},t_{-})$ where $t_+$ and $t_-$ respectively denote the dimensions of 
the positive and negative eigenvalues of $L\otimes \mathbb{R}$.
We define $\sign L:=t_+-t_-$.

A sublattice $M$ of a lattice $L$ is a submodule of $L$ with the bilinear form of $L$ restricted to $M$.
A sublattice $M$ of $L$ is called primitive if $L/M$ is torsion free. 
For a sublattice $M$ of $L$, we denote the orthogonal complement of $M$ in $L$ by  $M^{\bot}_L$ (or simply $M^\bot$). 
An action of a group $G$ on a lattice $L$ preserves the bilinear form unless otherwise stated and we define the invariant part $L^G$ and the coinvariant part $L_G$ of $L$ by
$$
L^G:=\{x\in L\ | \ g\cdot x=x \  (\forall g\in G) \},\quad
L_G:=(L^G)^\bot_L.
$$
We simply denote $L^{\langle g\rangle}$ and $L_{\langle g\rangle}$ by $L^{g}$ and $L_{g}$ respectively for $g \in G$.
If another group $H$ acts on $L$,
 we denote $L^G \cap L_H$ by $L^G_H$.

The hyperbolic lattice $U$ is the rank $2$ lattice whose Gram matrix is given by $\begin{bmatrix} 0 & 1\\ 1 & 0\\ \end{bmatrix}$ and 
the corresponding basis $e,f$ is called the standard basis. 
Let $A_{m}, \ D_{n}, \ E_{l}, \ (m\ge 1, \ n\ge 4, \ l =6,7,8)$ be the lattices defined by the corresponding Cartan matrices. 
Every indefinite even unimodular lattice can be realized as an orthogonal sum of copies of $U$ and $E_{8}(\pm 1)$ in an essentially unique way, 
the only relation being $E_{8}\oplus E_{8}(-1)\cong U^{\oplus 8}$.  
Thus an even unimodular lattice of signature $(3,19)$ is isomorphic to $\Lambda:= U^{\oplus 3} \oplus E_{8} (-1)^{\oplus 2}$, which is called the K3 lattice. 


Let $L$ be a non-degenerate even lattice. 
The bilinear form determined a canonical embedding $L \hookrightarrow L^{\vee}:=\mathrm{Hom}(L,\mathbb{Z})$. 
The discriminant group $A(L):=L^{\vee}/L$ is an abelian group of order $|\disc(L)|$. equipped with a quadratic map $q(L)\colon A(L) \rightarrow \mathbb{Q}/2\mathbb{Z}$
by sending $x+L \mapsto x^2 + 2\mathbb{Z}$. 
Two even lattices $L$ and $L'$ have isomorphic discriminant form 
if and only if they are stably equivalent, that is, $L\oplus K \cong L' \oplus K'$ for some even unimodular lattices $K$ and $K'$. 
Since the rank of an even unimodular is divisible by $8$,
 $\sign q(L):=\sign L \bmod 8$ is well-defined.  
Let $M \hookrightarrow L$ be a primitive embedding of non-degenerate even lattices and suppose that $L$ is unimodular, 
 then there is a natural isomorphism $(A(M),q(M))\cong (A(M^\perp),-q(M^\perp))$.  
The genus of $L$ is defined as the set of isomorphism classes of lattices $L'$ such that the signature of $L'$ is the same as that of $L$ and $(A(L),q(L)) \cong (A(L'),q(L'))$.

\begin{Thm}[\cite{Ni2,Om}] \label{Nik Genus}
Let $L$ be a non-degenerate even lattice with
 $\rk L \geq 3$.
If $L\cong U(n)\oplus L'$ for a positive integer $n$
 and a lattice $L'$,
 then the genus of $L$ consists of only one class. 
\end{Thm}

Let $L$ be a lattice and $M$ a module such that $L \subset M \subset L^{\vee}$. 
We say that $M$ equipped with the induced bilinear form $\langle *,**\rangle$ is an overlattice of $L$ if $\langle *,**\rangle$ takes integer values on $M$. 
Any lattice which includes $L$ as a sublattice of finite index is considered as an overlattice of $L$. 

\begin{Prop}[\cite{Ni2}] \label{Nik Isotropic}
Let $L$ be a non-degenerate even lattice and $M$ a submodule of $L^{\vee}$ such that $L \subset M$. 
Then $M$ is an even overlattice of $L$ if and only if the image of $M$ in $A(L)$ is an isotropic subgroup, that is, the restriction of $q(L)$ to $M/L$ is zero. 
Moreover, there is a natural one-to-one correspondence between the set of even overlattices of $L$ and the set of isotropic subgroups of $A(L)$. 
\end{Prop}


\begin{Prop}[\cite{Ni2}] \label{Nik Disc form}
Let $K$ and $L$ be non-degenerate even lattices. 
Then there exists a primitive embedding of $K$ into an even unimodular lattice $\Gamma$ such that $K^{\bot} \cong L$, 
if and only if $(A(K),q(K)) \cong (A(L),-q(L))$. 
More precisely, any such $\Gamma$ is of the form $\Gamma_{\lambda} \subset K^{\vee}\oplus L^{\vee}$ for some isomorphism
$$
\lambda\colon(A(K),q(K)) \rightarrow (A(L),-q(L)),
$$
where $\Gamma_{\lambda}$ is the lattice corresponding to the isotropic subgroup
$$
\{(x,\lambda(x))\in A(K)\oplus A(L) \ | \ x \in A(K) \} \subset A(K) \oplus A(L). 
$$
\end{Prop}

\begin{Lem} \label{Involution}
Let $L$ be a non-degenerate lattice and
 $\iota\in\mathrm{O}(L)$ an involution.   
Then
 $L/(L^\iota \oplus L_\iota) \cong (\Z/2\Z)^n$
 for some $n \le \min\{\rk L^\iota, \rk L_\iota\}$. 
\end{Lem}

\begin{proof}
For any $x \in L$, we have a decomposition $x=x_+ + x_-$ 
with $x_+ \in L^\iota \otimes \Q$ and $x_- \in L_ \iota \otimes \Q$. 
We have $2x_+=x + \iota (x) \in L$. 
We define $\phi(x \mod L^\iota \oplus L_\iota)=2x_+ \mod 2L^\iota$. 
We can easily see that $\phi\colon L/(L^\iota \oplus L_\iota) \rightarrow L^\iota/2L^\iota$ is a well-defined injection. 
Hence we have $L/(L^\iota \oplus L_\iota ) \cong (\Z/2\Z)^n$ with $n \le \rk L^\iota$. 
Similarly, we have $n \le \rk L_\iota$. 
\end{proof}


\subsection{K3 Surfaces} \label{SECT_k3}
A K3 surface $S$ is a simply-connected compact complex surface with trivial canonical bundle. 
Then $H^2(S,\Z)$ with its intersection form is isomorphic to the K3 lattice $\Lambda= U^{\oplus 3} \oplus E_{8} (-1)^{\oplus 2}$.
It is endowed with a weight-two Hodge structure
$$
H^2(S,\C)=
H^2(S,\Z)\otimes\C=H^{2,0}(S)\oplus H^{1,1}(S)\oplus H^{0,2}(S).
$$
Let $\omega_S$ be a nowhere vanishing holomorphic 2-form on $S$. 
The space $H^{2,0}(S)\cong\C$ is generated by the class of $\omega_S$, which we denote by the same $\omega_S$. 
The N\'eron--Severi lattice $\mathrm{NS}(S)$
 is given by
\begin{equation} \label{EQ_ns_lat}
\mathrm{NS}(S):=\{ x\in H^2(S,\Z) \bigm|
 \langle x,\omega_S \rangle=0 \}.
\end{equation}
Here we extend the bilinear form $\langle *,** \rangle$ on $H^2(S,\Z)$ to that on $H^2(S,\C)$ linearly.
The open subset $\mathcal{K}_S\subset H^{1,1}(S,\R)$ of K\"ahler classes is called the K\"ahler cone of $S$.
It is known that $\mathcal{K}_S$ is a connected component of
\begin{equation*}
 \{ x \in H^{1,1}(S,\R) \bigm|
  x^2>0, ~ \langle x,\delta \rangle\neq 0 ~
  (\forall \delta\in \Delta_S) \},\quad
 \Delta_S := \{ \delta \in \mathrm{NS}(S) \bigm| \delta^2=-2 \}.
\end{equation*}
The study of K3 surfaces reduces to lattice theory by the following two theorems.

\begin{Thm}[Global Torelli Theorem\ \cite{BHPV}] \label{Torelli}
Let $S$ and $T$ be $K3$ surfaces.
Let $\phi\colon H^2(S,\Z)\rightarrow H^2(T,\Z)$ be an isomorphism of lattices satisfying the following conditions.
\begin{enumerate}
\item \label{T_1}
$(\phi\otimes\C)(H^{2,0}(S))=H^{2,0}(T)$.
\item
There exists an element $\kappa\in\mathcal{K}_S$ such that $(\phi\otimes\R)(\kappa)\in\mathcal{K}_{T}$. 
\end{enumerate}
Then there exists a unique isomorphism $f\colon T\rightarrow S$ such that $f^*=\phi$.
\end{Thm}

\begin{Thm}[Surjectivity of the period map \cite{BHPV}] \label{SurjPed}
Assume that vectors $\omega\in\Lambda\otimes\C$ and $\kappa\in\Lambda\otimes\R$ satisfy the following conditions:
\begin{enumerate}
\item \label{cond_surj_1}
$ \langle \omega,\omega \rangle=0$, $ \langle\omega,\overline{\omega} \rangle>0$,
 $\langle \kappa,\kappa \rangle>0$ and
 $\langle \kappa,\omega \rangle=0$.
\item \label{cond_surj_2}
 $\langle \kappa,x \rangle\neq 0$
 for any $x\in(\omega)^\bot_\Lambda$ such that $\langle x,x \rangle=-2$.
\end{enumerate}
Then there exist a K3 surface $S$ and an isomorphism $\alpha\colon H^2(S,\Z)\rightarrow\Lambda$ of lattices such that
$\C\omega=(\alpha\otimes\C)(H^{2,0}(S))$ and $\kappa\in(\alpha\otimes\R)(\mathcal{K}_S)$.
\end{Thm}

An action of a group $G$ on a K3 surface $S$ induces a (left) $G$-action on $H^2(S,\Z)$ by
\begin{equation*} \label{defaction}
g\cdot x:=(g^{-1})^* x,\quad g\in G,~ x\in H^2(S,\Z).
\end{equation*}
The following lemma is useful to study finite group actions on a K3 surface.

\begin{Lem}[{\cite[Lemma 1.7]{OS}}]
 \label{LEM_reflection}
Let $S$ be a K3 surface with an action of a finite group $G$
 and let $x$ be an element in $\mathrm{NS}(S)^G\otimes\R$
 with $x^2>0$.
Suppose that $\langle x, \delta \rangle \neq 0$ for any $\delta\in \mathrm{NS}(S)$ with $\delta^2=-2$.
Then there exists $\gamma\in \mathrm{O}(H^2(S,\Z))$ such that $\gamma(H^{2,0}(S))=H^{2,0}(S)$, $\gamma(x)\in\mathcal{K}_S$, and $\gamma$ commutes with $G$.
\end{Lem}

An automorphism $g$ of a K3 surface $S$ is said to be symplectic if $g^*\omega_S=\omega_S$, and anti-symplectic if $g^* \omega_S=-\omega_S$.


\begin{Thm}[\cite{Ni}]\label{Nik FixPoint}
Let $g$ be a symplectic automorphism of $S$ of finite order, then $\mathrm{ord}(g)\le 8$ 
and the number of fixed points depends only on $\mathrm{ord}(g)$ as given in the following table. 
\begin{center}
 \begin{tabular}{|c|c|c|c|c|c|c|c|}  \hline
$\ord(g)$  & $2$  &  $3$  &  $4$  &  $5$ &  $6$  &  $7$  &  $8$\\ \hline
$|S^{g}|$  & $8$  &  $6$  &  $4$  &  $4$  &  $2$  &  $3$  &  $2$\\ \hline
 \end{tabular}
 \end{center} 
 \end{Thm}

A fixed point free involution $\iota$ of a K3 surface (necessarily anti-symplectic) is called an Enriques involution. 
The quotient surface $S/\langle \iota \rangle$ is called an Enriques surface. 
\begin{Thm}[{\cite[Section VIII.19]{BHPV}}] \label{thm_enriques_involution}
An involution $\iota$ of a K3 surface $S$ is an Enriques involution if and only if 
$$
 H^2(S,\Z)^\iota\cong U(2)\oplus E_8(-2),\quad
 H^2(S,\Z)_\iota\cong U\oplus U(2)\oplus E_8(-2).
$$
\end{Thm}

An example of a K3 surface with an Enriques involution we should keep in mind is: 

\begin{Ex}[{Horikawa model \cite[Section V.23]{BHPV}}] \label{K3->P1P1} 
The double covering $\pi\colon S \rightarrow \mathbb{P}^{1}\times \mathbb{P}^{1}$ branching along a smooth curve $B$ of bidegree $(4,4)$ is a K3 surface. 
We denote by $\theta$ the covering involution on $S$.  
Assume that $B$ is invariant under the involution $ \iota$ of $\PP^1 \times \PP^1$ given by $(x,y)\mapsto (-x,-y)$, where $x$ and $y$ are the inhomogeneous coordinates of $\PP^1\times \PP^1$.  
The involution $\iota$ lifts to a symplectic involution of the K3 surface $S$. 
Then $\theta \circ \iota$ is an involution of $S$ without fixed points unless $B$ passes through one of the four fixed points of $\iota$ on $\PP^1 \times \PP^1 $. 
The quotient surface $T=S/\langle \theta \circ \iota \rangle$ is therefore an Enriques surface. 
\[\xymatrix{
 \ar[d]_{/\langle \theta \rangle } S \ar[r]^{\id} & S  \ar[d]^{/ \langle \theta \circ \iota \rangle }\\
 \mathbb{P}^{1}\times \mathbb{P}^{1} & T \\
}\]
\end{Ex}

\begin{Prop}[{\cite[Proposition XIII.18.1]{BHPV}}] \label{PROP_horikawa}
Any generic K3 surface with an Enriques involution is realized as a Horikawa model defined above.
\end{Prop}

%


\section{Classification} 
The Bogomolov decomposition theorem \cite{Be1} implies that a Calabi--Yau threefold $X$ with infinite fundamental group admits an \'{e}tale Galois covering 
either by an abelian threefold or by the product of a K3 surface and an elliptic curve.   
We call $X$ of type A in the former case and of type K in the latter case. 
Among many candidates of such coverings, we can always find a unique smallest one, up to isomorphism as a covering \cite{Be2}. 
We call the smallest covering the minimal splitting covering of $X$. 
The goal of this section is to provide the following classification theorem of Calabi--Yau threefolds of type K:


\begin{Thm} \label{Main Thm}
There exist exactly eight Calabi--Yau threefolds of type K, up to deformation equivalence. 
The equivalence class is uniquely determined by the Galois group $G$ of the minimal splitting covering. 
Moreover, the Galois group is isomorphic to one of the following combinations of cyclic and dihedral groups
$$
C_{2},\ C_2 \times C_2, \ C_2 \times C_2 \times C_2, \ D_{6}, \ D_{8}, \ D_{10}, \ D_{12}, \ or \ C_2 \times D_8. 
$$
\end{Thm}

We will also give an explicit presentation for the eight Calabi--Yau threefolds. 
For the reader's convenience, we outline the proof of Theorem \ref{Main Thm}. 
Firstly, by the work of Oguiso and Sakurai \cite{OS}, the classification of Calabi--Yau threefolds of type K essentially reduces to 
that of K3 surfaces $S$ equipped with a Calabi--Yau action (Definition \ref{def_cyg_k3}) of a finite group of the form $G=H\rtimes \LF{\iota}$.
Here the action of $H$ on $S$ is symplectic and $\iota$ is an Enriques involution.
A sketch of the proof of the classification of such K3 surfaces is the following: 
\begin{enumerate}
\item We make a list of examples of K3 surfaces $S$ with a Calabi--Yau $G$-action. 
They are given as double coverings of $\mathbb{P}^1\times \mathbb{P}^1$ ($H$-equivariant Horikawa models).  
\item For a K3 surface S with a Calabi--Yau $G$-action, it is proven that there exists an element $v\in \mathrm{NS}(S)^G$ such that $v^2=4$ (Key Lemma). 
\item It is shown that $S$ has a projective model of degree $4$ and admits a $G$-equivariant double covering of a quadric hypersurface in $\PP^3$ isomorphic to $\PP^1\times\PP^1$ if it is smooth. 
Therefore, $S$ is generically realized as an $H$-equivariant Horikawa model constructed in Step 1.  
\item We classify the deformation equivalence classes of $S$ on a case-by-case basis and also exclude an unrealizable Galois group. 
\end{enumerate}
It is worth noting that a realization of a K3 surface $S$ with a Calabi--Yau $G$-action as a Horikawa model is in general not unique (Propositions \ref{PROP_triplet} and \ref{PROP_222section}).  

\subsection{Work of Oguiso and Sakurai}
We begin with a brief review of Oguiso and Sakurai's work \cite{OS}. 
Let $X$ be a Calabi--Yau threefold of type K.  
Then the minimal splitting covering $\pi\colon S \times E \rightarrow X$ is obtained by imposing the condition 
that the Galois group of the covering $\pi$ does not contain any elements of the form $(\id_{S},\text{non-zero translation of }E)$. 

\begin{Def}
We call a finite group $G$ a Calabi--Yau group if there exist a K3 surface $S$, an elliptic curve $E$ 
and a faithful $G$-action on $S\times E$ such that the following conditions hold.
\begin{enumerate}
\item $G$ contains no elements of the form $(\id_{S},\text{non-zero translation of }E)$.
\item The $G$-action on $H^{3,0}(S\times E)\cong \C$ is trivial.
\item The $G$-action is free, that is, $(S\times E)^g=\emptyset$ for all $g \in G$, $g\neq 1$.
\item $G$ does not preserve any holomorphic 1-form, that is, $H^{0}(S\times E,\varOmega_{S\times E})^{G}=0$.
\end{enumerate}
We call $S\times E$ a target threefold of $G$. 
\end{Def}

The Galois group $G$ of the minimal splitting covering $S\times E\rightarrow X$ of a Calabi--Yau threefold $X$ of type K is a Calabi--Yau group.
Conversely, if $G$ is a Calabi--Yau group with a target space $S\times E$ of $G$, then the quotient $(S\times E)/G$ is a Calabi--Yau threefold of type K. 

Let $G$ be a Calabi--Yau group and $S\times E$ a target threefold of $G$. 
Thanks to a result of Beauville \cite{Be2}, we have a canonical isomorphism $\Aut(S\times E)\cong \Aut(S)\times \Aut(E)$. 
The images of $G\subset \Aut(S\times E)$ under the two projections to $\Aut(S)$ and $\Aut(E)$ are denoted by $G_{S}$ and $G_{E}$ respectively.  
It can be proven that $G_{S}\cong G \cong G_{E}$ via the natural projections:
\[\xymatrix{
\Aut(S) \ar@{}[d]|\bigcup & \ar[l]_{p_1}\Aut(S \times E) \ar@{}[d]|\bigcup \ar[r]^{p_2} & \Aut(E) \ar@{}[d]|\bigcup \\
G_S & \ar[l]^{p_1|_{G}}_{\cong} G \ar[r]_{p_2|_{G}}^{\cong} &  G_E. 
}\]
We denote by $g_S$ and $g_E$ the elements in $G_S$ and $G_E$ respectively corresponding to $g \in G$, 
that is, $p_1(g)=g_S$, $p_2(g)=g_E$. 

\begin{Prop}[{\cite[Lemma 2.28]{OS}}] \label{PROP_OS}
Let $G$ be a Calabi--Yau group and $S\times E$ a target threefold of $G$. 
Define $H:=\Ker(G\rightarrow \GL(H^{2,0}(S)))$ and take any $\iota \in G \setminus H$. 
Then the following hold.
\begin{enumerate}
\item $\ord(\iota)=2$ and $G=H\rtimes \langle \iota \rangle$, where the semi-direct product structure is given by $\iota h\iota=h^{-1}$ for all $h\in H$.
\item $g_S$ is an Enriques involution
for any $g\in G\setminus H$.
\item $\iota_{E}=-\id_{E}$ and $H_{E}=\langle t_{a}\rangle \times \langle t_{b}\rangle\cong C_{n}\times C_{m}$ under an appropriate origin of $E$, 
where $t_a$ and $t_b$ are translations of order $n$ and $m$ respectively such that $n|m$. 
Moreover we have $(n,m)\in \{(1,k) \, (1\le k \le 6), \ (2,2), \ (2,4), \ (3,3)\}$. 
\end{enumerate}
\end{Prop}

Although the case $(n,m)=(2,4)$ is eliminated from the list of possible Calabi--Yau groups in \cite{OS}, 
there is an error in the proof of Lemma 2.29 in \cite{OS}, which is used to prove the proposition above\footnote{
The error in \cite[Lemma 2.29]{OS} is that, with their notation, the group $\langle \alpha, h_S^2\rangle$ is not necessarily isomorphic to either $C_2\times C_2$ or $C_2 \times C_4$, 
but may be isomorphic to $C_4$.}. 
In fact, there exists a Calabi--Yau group of the form $(C_{2}\times C_{4})\rtimes C_2 \cong C_{2}\times D_{8}$ (Proposition \ref{PROP_triplet}). 
For the sake of completeness, here we settle the proof of Lemma 2.29 in \cite{OS}. 
We do not repeat the whole argument but give a proof of the non-trivial part: $(n,m)$ cannot be $(1,7),(1,8),(2,6)$ nor $(4,4)$. 
The reader can skip this part, assuming Proposition \ref{PROP_OS}. 
\begin{proof}[Proof of $(n,m)\ne(1,7),(1,8),(2,6),(4,4)$] \
\begin{enumerate}
\item $(1,7)$ : 
Let $H_{S}=\langle g\rangle \cong C_{7}$. 
Since $\iota g \iota=g^{-1}$, $\langle \iota \rangle \cong C_{2}$ acts on $S^{g}$, which has cardinality $3$ and thus has a fixed point. 
But this contradicts with the fact that $S^{f}=\emptyset$ for any $f \in G_{S}\setminus H_{S}$.  

\item $(1,8)$ : 
Let $H_{S}=\langle g\rangle \cong C_{8}$. 
Note that $\langle g, \iota \rangle/\langle g^{2}\rangle \cong C_{2}\times C_2$ and acts on $S^{g^{2}}\setminus S^{g}$ which has cardinality $4-2=2$. 
Then this action induces a homomorphism $\phi \colon C_{2} \times C_2\rightarrow S_{2}$. 
Since $S^{f}=\emptyset$ for all $f \in G_{S}\setminus H_{S}$, 
$\Ker(\phi) (\ne 1)\subset \langle g\rangle /\langle g^{2}\rangle \cong C_{2}$ and thus $\Ker(\phi)\cong \langle g\rangle /\langle g^{2}\rangle$. 
This contradicts with our subtracting $S^{g}$ from $S^{g^{2}}$. 

\item $(2,6)$ : 
Let $H_{S}=\langle g\rangle \times \langle h\rangle \cong C_{2}\times C_{6}$. 
$|S^{h}|=2$ and thus there is a homomorphism $\phi\colon  \langle g, \iota \rangle \cong C_{2} \times C_2\rightarrow S_{2}$. 
Since $S^{f}=\emptyset$ for all $f \in G_{S}\setminus H_{S}$, $\Ker(\phi)=\langle g \rangle \cong C_{2}$. 
Let $p \in S^{g}$ be one of the fixed points. Then we have a faithful representation $H_{S}=C_{2}\times C_{6}\rightarrow \SL(T_{p}S) \cong \SL(2,\C)$. 
This contradicts with the classification of finite subgroups in $\SL(2,\C)$. 

\item $(4,4)$ : 
Let $H_{S}=\langle g\rangle \times \langle h\rangle \cong C_{4} \times C_4$. 
Note that $\langle g,h,\iota \rangle/\langle g^{2},h^{2}\rangle \cong C_{2}\times C_2 \times C_2$ and acts on $S^{h^{2}}$. 
Note also that $|S^{h^{2}}\setminus S^{h}|=8-4=4$. 
Then this induces a homomorphism $\phi\colon C_{2} \times C_2 \times C_2\rightarrow S_{4}$. 
Since $S_{4}$ does not contain $C_{2} \times C_2 \times C_2$, $\Ker(\phi)\subset \langle g,h\rangle/\langle g^{2},h^{2}\rangle$ is not trivial. 
Let $\alpha$ be a lift of a non-trivial element of $\Ker(\phi)$ and take a fixed point $p \in S^{h^{2}}\setminus S^{h}$. 
Then we have a natural injection $\langle \alpha, h^{2}\rangle \rightarrow \SL(T_{p}S) \cong \SL(2,\C)$. 
In addition using $h\notin \Ker(\phi)$, we obtain $\langle \alpha, h^{2}\rangle \cong C_{4}\times C_{2}$, 
which contradicts with the classification of finite subgroups in $\SL(2,\C)$. 
\end{enumerate}
\end{proof}

Now we can state a main result of \cite{OS} with a slight correction. 

\begin{Thm}[{\cite[Theorem 2.23]{OS}}]
Let $X$ be a Calabi--Yau threefold of type $K$. 
Let $S\times E\rightarrow X$ be the minimal splitting covering and $G$ its Galois group. 
Then the following hold. 
\begin{enumerate}
\item $G$ is isomorphic to one of the following: 
$$
C_{2}, \ C_2 \times C_2, \ C_2 \times C_2 \times C_2 , \
D_{6}, \ D_8, \ D_{10}, \ D_{12}, \ C_{2}\times D_{8}, \ \text{or} \ (C_{3}\times C_3)\rtimes C_{2}.
$$ 
\item In each case the Picard number $\rho(X)$ of $X$ is uniquely determined by $G$ and is calculated as $\rho(X)=11,7,5,5,4,3,3,3,3$ respectively. 
\item The cases $G\cong C_{2}, \ C_2 \times C_2, \ C_2 \times C_2 \times C_2 $ really occur. 
\end{enumerate}
\end{Thm}
It has not been settled yet whether or not there exist Calabi--Yau threefolds of type K 
with Galois group $G\cong D_{2n}\ (3\le n\le 6)$, $C_{2}\times D_{8}$ or $(C_{3}\times C_3)\rtimes C_{2}$. 
Note that the example of a Calabi--Yau threefold of type K with $G\cong D_{8}$ presented in Proposition 2.33 in \cite{OS} is incorrect\footnote{
The error in \cite[Proposition 2.33]{OS} is that, with their notation, $S^{ab}\ne \emptyset$.}. 
We will settle this classification problem of Calabi--Yau threefolds of type K and also give an explicit presentation of the deformation classes. 

\begin{Ex}[Enriques Calabi--Yau threefold] \label{Enr CY3}
Let $S$ be a K3 surface with an Enriques involution $\iota$ and $E$ an elliptic curve with the negation $-1_E$.  
The free quotient 
$$
X:=(S\times E)/\langle (\iota,-1_E) \rangle
$$
is the simplest Calabi--Yau threefold of type K, known as the Enriques Calabi--Yau threefold.
\end{Ex}


\subsection{Construction} \label{Construction}
The goal of this section is to make a list of concrete examples of Calabi--Yau threefolds of type K. 
We will later show that the list covers all the generic Calabi--Yau threefolds of type K. 
We begin with the definition of  {\it Calabi--Yau actions}, which is based on Proposition \ref{PROP_OS}. 

\begin{Def} \label{def_cyg_k3}
Let $G$ be a finite group. 
We say that an action of $G$ on a K3 surface $S$ is a Calabi--Yau action if the following hold.
\begin{enumerate}
\item $G=H\rtimes \langle \iota \rangle$ for some $H\cong C_n\times C_m$ with $(n,m)\in \{(1,k)\ (1\le k \le 6), \ (2,2), \ (2,4), \ (3,3)\}$, and $\iota$ with $\ord(\iota)=2$.  
The semi-direct product structure is given by $\iota h\iota=h^{-1}$ for all $h\in H$.
\item $H$ acts on $S$ symplectically and any $g\in G\setminus H$ acts on $S$ as an Enriques involution.
\end{enumerate}
\end{Def}

Recall that a generic K3 surface with the simplest Calabi--Yau action, namely an Enriques involution,
is realized as a Horikawa model (Proposition \ref{PROP_horikawa}).
We will see that a K3 surface equipped with a Calabi--Yau $G$-action is realized as an $H$-equivariant Horikawa model (Proposition \ref{PROP_triplet}). 

We can construct a Calabi--Yau $G$-action on a K3 surface as follows.
Let $u,x,y,z,w$ be affine coordinates of $\C\times\C^2\times\C^2$.
Define $L$ by
$$
 L:=\left(\C\times(\C^2\setminus \{0\})\times(\C^2\setminus \{0\})\right)/(\C^\times)^2,
$$
 where the action of $(\mu_1,\mu_2)\in(\C^\times)^2$ is given by
$$
 (u,x,y,z,w)\mapsto
 (\mu_1^2 \mu_2^2 u,\mu_1 x,\mu_1 y,\mu_2 z,\mu_2 w).
$$
The projection $\C\times\C^2\times\C^2\rightarrow \C^2\times\C^2$ descends to the map
\begin{equation*}
 \pi \colon L\rightarrow
 Z:=\left((\C^2\setminus \{0\})\times(\C^2\setminus \{0\})\right)
  /(\C^\times)^2
 \cong\PP^1\times\PP^1.
\end{equation*}
Note that $L$ is naturally identified with the total space of $\mathcal{O}_Z(2,2)$.
Let $F=F(x,y,z,w)\in H^0(\mathcal{O}_Z(4,4))$ be a homogeneous polynomial of bidegree $(4,4)$.
Assume that the curve $B\subset Z$ defined by $F=0$ has at most ADE-singularities.
We define $S_0$ by
$$
 S_0:=\{u^2=F\}\subset L.
$$
In other words, $S_0$ is a double covering of $Z$ branching along $B$. 
The minimal resolution $S$ of $S_0$ is a $K3$ surface (see the proof of Lemma \ref{LEM_detdet} below).
The group $\Gamma:=\GL(2,\C)\times \GL(2,\C)$ acts on $L$ by, for $\gamma=M_1\times M_2\in\Gamma$, 
$$
 \gamma (u,x,y,z,w)=
 (u,x',y',z',w'), \
 \begin{bmatrix} x' \\ y' \end{bmatrix}
 = M_1 \begin{bmatrix}x \\ y\end{bmatrix}, \
 \begin{bmatrix} z' \\ w' \end{bmatrix}
 = M_2 \begin{bmatrix}z \\ w\end{bmatrix}. 
$$ 
If $F$ is invariant under the action of $\gamma\in\Gamma$, then $\gamma$ naturally acts on $S_0$ as well. 
We denote by $\gamma^+$ the induced action of $\gamma$ on $S$.
The covering transformation of $S_0\rightarrow Z$, which is defined by $(u,x,y,z,w)\mapsto(-u,x,y,z,w)$, induces the involution $j:=(\sqrt{-1}\,I_2\times I_2)^+$ on $S$.
Note that there are two lifts of the action of $\gamma$ on $Z$: $\gamma^+$ and $\gamma^-:=\gamma^+ j$.
%
%
The K3 surface $S$ with the polarization $\mathcal{L}:=\pi^*\mathcal{O}_Z(1,1)$ is a polarized K3 surface of degree $4$.
Let $\Aut(S,\mathcal{L})$ denote the automorphism group of the polarized K3 surface $(S,\mathcal{L})$. 
\begin{Def} \label{(S,L,G)}
We denote by $(S,\mathcal{L},G)$ a triplet consisting of a polarized K3 surface $(S,\mathcal{L})$ defined above and a finite subgroup $G \subset \Aut(S,\mathcal{L})$. 
\end{Def}
We define $\Aut(S,\mathcal{L})^+$ to be the subgroup of $\Aut(S,\mathcal{L})$ preserving each ruling $Z\rightarrow \PP^1$. 
Then we have
\begin{equation} \label{EQ_polarizedaut} 
\Aut(S,\mathcal{L})^+=
 \{ \gamma\in\Gamma \bigm| \gamma^*F=F \}/
 \{ \lambda_1 I_2 \times \lambda_2 I_2 \bigm|
 \lambda_1^2 \lambda_2^2=1 \}.
\end{equation}

\begin{Rem}
Since the Picard group of $Z$ is isomorphic to $\Z^{\oplus 2}$ (hence torsion free), 
a line bundle $\mathcal{M}$ on $Z$ such that $\mathcal{M}^{\otimes 2}\cong\mathcal{O}_Z(4,4)$ is isomorphic to $\mathcal{O}_Z(2,2)$.
Therefore, a double covering of $Z$ branching along $B$ is unique and isomorphic to $S_0$.
\end{Rem}


\begin{Lem} \label{LEM_detdet}
Let $\omega_S$ be a nowhere vanishing holomorphic 2-form on $S$. 
If $F$ is invariant under the action of $\gamma=M_1\times M_2\in \Gamma$, 
we have
$$
 (\gamma^\pm)^* \omega_S=\pm\det(M_1)\det(M_2) \omega_S.
$$
\end{Lem}
\begin{proof}
Since $(xdy-ydx)\wedge(zdw-wdz)/u$ gives a nowhere vanishing holomorphic $2$-form on $S$, the equality in the assumption holds.
\end{proof}

\begin{Lem} \label{LEM_fixedpoint}
Let $\phi \colon Y\rightarrow Y_0$ be the minimal resolution of a surface $Y_0$ with at most ADE-singularities.
Then an automorphism $g$ of $Y_0$ has a fixed point if and only if the induced action of $g$ on $Y$ has a fixed point.
\end{Lem}
\begin{proof}
The assertion follows from the fact that any automorphism of a connected ADE-configuration has a fixed point.
\end{proof}


\begin{Prop} \label{PROP_triplet}
Let $(S,\mathcal{L},G)$ be a triplet defined in Definition \ref{(S,L,G)}. 
Assume that the action of $G$ on $S$ is a Calabi--Yau action.
Then such triplets $(S,\mathcal{L},G)$ are classified into the types in the following table up to isomorphism.
Here two triplets $(S,\mathcal{L},G)$ and $(S',\mathcal{L}',G')$ are isomorphic 
if there exists an isomorphism $f\colon S\rightarrow S'$ such that $f^*\mathcal{L}'=\mathcal{L}$ and $f^{-1} \circ G' \circ f=G$.
\begin{equation*}
\begin{array}{|c|c|c|} \hline
 H & \Xi & \text{basis of } H^0(\mathcal{O}_Z(4,4))^G \\
 \hline
 C_1 & \emptyset &  x^i y^{4-i} z^j w^{4-j} + x^{4-i} y^i z^{4-j} w^j \\ \hline 
 C_2 & \{(1/4,1/4) \}&  x^i y^{4-i} z^j w^{4-j} + x^{4-i} y^i z^{4-j} w^j \ \  (i \equiv j \bmod 2) \\ \hline 
 C_2 & \{(1/4,0) \}&  x^i y^{4-i} z^j w^{4-j} + x^{4-i} y^i z^{4-j} w^j \ \ (i \equiv 0 \bmod 2) \\ \hline 
 C_3 & \{(1/3,1/3) \}& x^i y^{4-i} z^j w^{4-j} + x^{4-i} y^i z^{4-j} w^j \ \ ( i+j\equiv 1 \bmod 3) \\ \hline 
 C_4 & \{(1/8,1/8) \}&   x^4 z^4+y^4 w^4,x^4 w^4+y^4 z^4,x^3 y z w^3+x y^3 z^3 w, x^2 y^2 z^2 w^2 \\ \hline 
 C_4 & \{(1/8,1/4) \}&  x^4 z^3 w+y^4 z w^3,x^4 z w^3+y^4 z^3 w,x^2 y^2 z^4+x^2 y^2 w^4,  x^2 y^2 z^2 w^2 \\ \hline 
 C_5 & \{(1/5,2/5) \}& x^4 z w^3+y^4 z^3 w,x^3 y z^4+x y^3 w^4,x^2 y^2 z^2 w^2 \\ \hline 
 C_6 & \{(1/12,1/6) \}& x^4 z^4+y^4 w^4,x^4 z w^3+y^4 z^3 w,x^2 y^2 z^2 w^2 \\ \hline 
 C_2 \times C_2& \{(1/4,0),(0,1/4) \}& x^i y^{4-i} z^j w^{4-j} + x^{4-i} y^i z^{4-j} w^j\ \ ( i\equiv j \equiv 0 \bmod 2)\\ \hline 
 C_2 \times C_4 & \{(1/8,1/8),(0,1/4)\}& x^4 z^4+y^4 w^4,x^4 w^4+y^4 z^4,x^2 y^2 z^2 w^2 \notag  \\ \hline 
\end{array}
\end{equation*}
A triplet $(S,\mathcal{L},G)$ for each type is defined as follows. 
\begin{enumerate}
\item
$F$ is invariant under the action of
 $M(a)\times M(b)$ for all $(a,b)\in \Xi$ and
 $\iota_1$, where
 $$
 M(a):=
 \begin{bmatrix} \exp(2\pi i a) & 0 \\
0 & \exp(-2\pi i a) \end{bmatrix} \ ,
\iota_1:=
 \begin{bmatrix}
        0 & 1\\
        1 & 0\\
        \end{bmatrix}
        \times
 \begin{bmatrix}
        0 & 1\\
        1 & 0\\
        \end{bmatrix}\text{.}
 $$
\item
$H:=\langle (M(a)\times M(b))^+ \bigm| (a,b)\in \Xi \rangle$.
\item
$G:=H\rtimes \langle \iota \rangle, \ 
 \iota:=\iota_1^-
 = \left(
 \sqrt{-1}  \begin{bmatrix}
        0 & 1\\
        1 & 0\\
        \end{bmatrix}
        \times
 \begin{bmatrix}
        0 & 1\\
        1 & 0\\
        \end{bmatrix}
 \right)^+ \text{.}$
\item \label{ITEM_nofixed}
for any $g\in G\setminus H$, the action of $g$ on $B$ has no fixed point.
\end{enumerate}
Furthermore, for a generic $F \in H^0(\mathcal{O}_Z(4,4))^G$,
 the surface $S_0$ is smooth and the condition (\ref{ITEM_nofixed}) is satisfied, and thus the action of $G$ on $S_0(=S)$ is a Calabi--Yau action.
\end{Prop}
\begin{Rem} \label{REM_invariant}
In Proposition \ref{PROP_triplet}, the group $G\subset \Aut(S,\mathcal{L})$ for each type acts on
$H^0(\mathcal{O}_Z(4,4))$ in a natural way by using the generator matrices.
Hence we can define the $G$-invariant space $H^0(\mathcal{O}_Z(4,4))^G$. 
We will use a similar convention in Proposition \ref{PROP_222section}.
\end{Rem}

\begin{proof}
Let $(S,\mathcal{L},G)$ be a triplet such that the action of $G=H\rtimes\LF{\iota}$ on $S$ is a Calabi--Yau action.
By Definition \ref{def_cyg_k3}, $H$ is isomorphic to one in the table or $C_3 \times C_3$.
For $g\in G\setminus H$,  the action of $g$ preserves each ruling $Z\rightarrow \PP^1$; otherwise, $Z^g$ is $1$-dimensional and $S^g\neq \emptyset$.
Since $G$ is generated by $G\setminus H$, we may assume that any element in $G$ is of the form $\gamma^{\pm}$ for $\gamma\in\Gamma$ with $\gamma^* F=F$ by (\ref{EQ_polarizedaut}).
By Lemma \ref{LEM_fixedpoint}, the condition (\ref{ITEM_nofixed}) is satisfied.

We begin with the case $H=C_1$.
Since $S^\iota=\emptyset$, it follows that $Z^\iota$ is (at most) $0$-dimensional.
Hence we may assume that
 $$\iota=
 \left(
  \lambda\begin{bmatrix}
        0 & 1\\
        1 & 0\\
        \end{bmatrix}
        \times
 \begin{bmatrix}
        0 & 1\\
        1 & 0\\
        \end{bmatrix}
 \right)^+$$
 after changing the coordinates of $Z$.
Since the action of $\iota$ on $S$ is anti-symplectic by Theorem \ref{thm_enriques_involution}, we have $\lambda=\pm \sqrt{-1}$ by Lemma \ref{LEM_detdet}, thus $\iota=\iota_1^-$.

Let us next consider the case $H=C_n$ ($1\leq n\leq 6$). 
Let $\sigma$ be a generator of $H$.
By the argument above and the relation $\iota\sigma\iota=\sigma^{-1}$, 
we may assume that $\iota=\iota_1^-$ and $\sigma=(\lambda M(a)\times M(b))^+$ for some $a,b\in \Q$ after changing the coordinates of $Z$.
Since the action of $\sigma$ on $S$ is symplectic, we have $\lambda=\pm 1$ by Lemma \ref{LEM_detdet}, thus $\sigma=(M(a)\times M(b))^+$.
If $ka\not\in \frac{1}{4}\Z$ and $kb\in \frac{1}{2}\Z$ for some $k\in\Z$, then $F$ is divisible by $x^2y^2$.
Hence this case is excluded.
We can see that $(a,b)$, $(b,a)$, $(-a,b)$, $(a+1/2,b)$, and $(ka,kb)$ with GCD$(k,n)=1$ give isomorphic triplets.
Therefore, we may assume that $(a,b)$ is one of the following:
\begin{align}
 (1/4,1/4),\ (1/4,0), \ (1/3,&1/3), \ (1/8,1/8), \ (1/8,1/4),\notag \\
 (1/5,1/5), \ (1/5,2/5), \ &(1/12,1/12),\ (1/12,1/6).\notag 
\end{align}
Here we have $n=\operatorname{min}\{k \in \Z_{>0} \ | \ ka\in\frac{1}{2}\Z,kb\in\frac{1}{2}\Z \}$.
Suppose that $(a,b)=(1/5,1/5)$ or $(1/12,1/12)$. Then $F$ is a linear combination of
$$
 x^4 w^4+y^4 z^4,~ x^3yzw^3+xy^3z^3w,~ x^2y^2z^2w^2,
$$
and hence $B$ has a singular point of multiplicity $4$ at, for instance,\ $[1:0]\times[1:0]$.
This contradicts to the assumption that $B$ admits only ADE-singularities.
Therefore the cases $(a,b)=(1/5,1/5)$ and $(1/12,1/12)$ are excluded.

Lastly, let us consider the cases $H=C_2\times C_2,\ C_2\times C_4$, and $C_3 \times C_3$. 
Let $\sigma,\tau$ be generators of $H$ such that $\ord(\sigma)$ is divisible by $\ord(\tau)$.
By a similar argument for $H=C_n$, $n=\ord(\sigma)$, we may assume that
$$
 \sigma=(M(a)\times M(b))^+,~
 \tau=\left(
  \lambda M(a')
  \begin{bmatrix}
        0 & 1\\
        1 & 0\\
        \end{bmatrix}^e
  \times
  M(b')
  \begin{bmatrix}
        0 & 1\\
        1 & 0\\
        \end{bmatrix}^f
 \right)^+,~
 \iota=\iota_1^-,
$$
 where $e,f\in\{0,1\}$, $\lambda^2 (-1)^{e+f}=1$ and $a,b,a',b'\in\Q$.
We can chose $(a,b)$ as is given in the table for $H=C_n$. 
Moreover, we may assume that $a'=0$ after replacing $\tau$ by $\sigma^k \tau$ for some $k \in \Z$.
By the argument for $H=C_1$, it follows that $Z^{h \iota}$ is $0$-dimensional for $h\in H$.
This implies that $e=0$, and $f=0$ if $(a,b)\neq(1/4,0)$.
We may assume that one of the following cases occurs.
\newcommand{\JJ}
 {\left( \begin{smallmatrix}&1\\1&\end{smallmatrix} \right)}
$$
\begin{array}{|c|c|c|c|} \hline
 & (a,b) & (a',b') & (e,f) \\
 \hline
 \text{(i)}   & (1/4,0) & (0,1/4) & (0,0) \\ \hline
 \text{(ii)}  & (1/4,0) & (0,1/4) & (0,1) \\ \hline
 \text{(iii)} & (1/8,1/8) & (0,1/4) & (0,0) \\ \hline
 \text{(iv)}  & (1/8,1/4) & (0,1/4) & (0,0) \\ \hline
 \text{(v)}   & (1/3,1/3) & (0,1/3) & (0,0) \\  \hline
\end{array}
$$
We can check that the cases (i) and (ii) for $H=C_2\times C_2$
 give isomorphic triplets.
Since $F$ is divisible by $x^2 y^2$ in the cases (iv) and (v),
 these cases are excluded.

Conversely, we check that the action of $G$ on $S$ is a Calabi--Yau action for $(S,\mathcal{L},G)$ of each type.
By the argument above, the action of $H$ on $S$ is symplectic.
Let $g\in G \setminus H$.
Then $B^g=\emptyset$ and $Z^g$ is $0$-dimensional, thus $S^g$ is either empty or $0$-dimensional.
Recall that the fixed locus of an anti-symplectic involution of a K3 surface is $1$-dimensional if it is not empty (see Section \ref{SECT_k3}).
This implies that $S^g=\emptyset$.
Therefore, the action of $G$ on $S$ is a Calabi--Yau action.

To show the smoothness of $S_0$ for a generic $F$, it suffices, by Bertini's theorem, to show that $B$ is smooth on the base locus of the linear system defined by $H^0(\mathcal{O}_Z(4,4))^G$.
We can check this directly.
We also find that the action of any $g\in G \setminus H$ on $B$ has no fixed point for a generic $F$.
\end{proof}

For $H=C_2$ or $C_4$, we obtain two families of K3 surfaces with a Calabi--Yau action of $G=H\rtimes \langle\iota\rangle$ by Proposition \ref{PROP_triplet}. 
As we will see in Proposition \ref{PROP_222section} below, if we forget the polarizations, they are essentially and generically the same families of K3 surfaces with a Calabi--Yau $G$-action. 
Let
$$
 Q=Q(s_1,t_1,s_2,t_2,s_3,t_3)
  \in H^0(\mathcal{O}_{\PP}(2,2,2)),\quad
 \PP:=\PP^1\times\PP^1\times\PP^1,
$$
 be a homogeneous polynomial of tridegree $(2,2,2)$, where $[s_i:t_i]$ is homogeneous coordinates of the $i$-th $\PP^1$.
Assume that the surface
\begin{equation*}
 S_0:=\{Q=0\}\subset \PP
\end{equation*}
has at most ADE-singularities. 
Then the minimal resolution $S$ of $S_0$ is a K3 surface. 
Let $p_i \colon \PP\rightarrow \PP^1$ denote the $i$-th projection.
With this notation, we have the following.

\begin{Prop} \label{PROP_222section}
Let $G\subset \PGL(2,\C)^3$ be the group defined by $\Xi'$ in the following table, in a similar way to Proposition \ref{PROP_triplet}.
\begin{equation*}
\begin{array}{|c|c|c|}
 \hline
 H & \Xi' & \text{basis of } H^0(\mathcal{O}_\PP(2,2,2))^G \\ \hline
 C_1 & \emptyset & \invpoly  \\ \hline 
 C_2 & \{(1/4,1/4,0) \}
  & \begin{matrix} \invpoly \\ (i+j \equiv 0 \bmod 2) \end{matrix}
  \\ \hline 
 C_3 & \{(1/3,1/3,1/3) \}
  & \begin{matrix} \invpoly \\ (i+j+k \equiv 0 \bmod 3) \end{matrix}
  \\ \hline 
 C_4 & \{(1/8,1/8,1/4) \}
  & \begin{matrix} \invpoly \\ (i+j+2k \equiv 0 \bmod 4) \end{matrix}
  \\ \hline 
 C_2 \times C_2& \{ (1/4,0,1/4),(0,1/4,1/4) \}
  & \begin{matrix} \invpoly \\ (i+k \equiv j+k \equiv 0 \bmod 2) \end{matrix}
  \\ \hline 
\end{array}
\end{equation*}
More precisely, $H$ is generated by $M(a)\times M(b)\times M(c)$ for $(a,b,c)\in \Xi'$, and $G:=H\rtimes \langle \iota \rangle$, where
\begin{equation*}
 \iota=\begin{bmatrix}
        0 & 1\\
        1 & 0\\
        \end{bmatrix}
         \times \begin{bmatrix}
        0 & 1\\
        1 & 0\\
        \end{bmatrix}
 \times \begin{bmatrix}
        0 & 1\\
        1 & 0\\
        \end{bmatrix}.
\end{equation*}
For a generic $Q\in H^0(\mathcal{O}_\PP(2,2,2))^G$, 
the surface $S_0$ is smooth and the action of $G$ on $S(=S_0)$ is a Calabi--Yau action.
Moreover, a generic triplet for $H=C_n \ (1\leq n \leq 4)$ or $C_2 \times C_2$ constructed in Proposition \ref{PROP_triplet} is of the form $(S,\mathcal{L},G)$, 
where
 $$
 \mathcal{L}=((p_\alpha\times p_\beta)^* \mathcal{O}_{\PP^1\times\PP^1}(1,1))|_S
 $$
 with $\alpha,\beta \in \{ 1,2,3 \}, \ \alpha \ne \beta$.
\end{Prop}
\begin{proof}
In each case, we can check, for a generic $Q\in H^0(\mathcal{O}_\PP(2,2,2))^G$, that $S_0$ is smooth and that the action of any $g\in G\setminus H$ on $S(=S_0)$ has no fixed point 
 by a direct computation (see the proof of Proposition \ref{PROP_triplet}).
Since the action of $H$ on $S$ is symplectic (Lemma \ref{LEM_detdet}), the action of $G$ on $S$ is a Calabi--Yau action.
This proves the first assertion.

Let us next show the second assertion.
We assume that $(\alpha,\beta)=(1,2)$ as the other cases are similar.
Define the map
\begin{equation*}
 \phi \colon
 V:=H^0(\mathcal{O}_{\PP}(2,2,2))^G
 \rightarrow
 W:=H^0(\mathcal{O}_{\PP}(4,4,0))^G
\end{equation*}
 by
\begin{equation*}
 Q=A s_3^2+B s_3 t_3+C t_3^2
 \mapsto F= \det \begin{bmatrix} A & B/2 \\ B/2 & C \end{bmatrix}.
\end{equation*}
The branching locus $B$ of the double covering $p_1\times p_2\colon S_0\rightarrow \PP^1\times\PP^1$ is defined by $F=0$.
The map $\phi$ gives a correspondence between $(2,2,2)$-hypersurfaces in $\PP$ with a Calabi--Yau $G$-action 
and double coverings of $\PP^1\times\PP^1$ branching along a $(4,4)$-curve with a Calabi--Yau $G$-action.
Hence it suffices to show that $\phi$ is dominant.
We show this by comparing the dimensions of $\phi^{-1}(F)$, $V$ and $W$.
By the argument above and generic smoothness (\cite[Corollary III.10.7]{H}), we may assume that $S_0$ and $\phi^{-1}(F)$ is smooth (hence $S=S_0$) by taking a generic $Q\in V$.
Let $\Delta$ be a contractible open neighborhood of $Q$ in $\phi^{-1}(F)$.
We construct a natural family of embeddings $S\rightarrow \PP$ parametrized by $\Delta$ as follows.
Since the branching locus of the double covering
\begin{equation*}
 {(p_1\times p_2)\times \id_\Delta}\colon
 \overline{S}\rightarrow (\PP^1\times\PP^1)\times \Delta,\quad
 \overline{S}:=\{ (x,Q') \bigm| Q'(x)=0 \}\subset \PP\times\Delta,
\end{equation*}
 is $B\times \Delta$,
 we have a natural commutative diagram
\begin{equation*}
\xymatrix{
 S\times \Delta \ar[r]^f_{\cong} \ar[rd]_(0.35){(p_1\times p_2)\times \id_\Delta \quad}
  & \overline{S} 
    \ar[d]^{(p_1\times p_2)\times \id_\Delta} \\
   & (\PP^1 \times \PP^1) \times \Delta,
}
\end{equation*}
 where 
\begin{equation*}
 f_{Q'}:=f(*,Q')\colon S \rightarrow
  \{Q'=0\}, \quad
 Q'\in\Delta,
\end{equation*}
 is a $G$-equivariant isomorphism and $f_Q=\id_S$.
Since $\Delta$ is connected and the Picard group of $S$ is discrete, the map $f_{Q'}$ is represented by $\gamma\in \GL(2,\C)^3$ such that $\gamma$ commutes with $G$ in $\PGL(2,\C)^3$.
Furthermore, since $(p_1\times p_2)\circ f_{Q'}=p_1\times p_2$, we may assume that $\gamma=I_2\times I_2\times M$.
Then $\phi(\gamma^* Q)=\det(M)^2 F$.
Therefore, we have $\Delta \subset \Gamma^* Q$, where $\Gamma$ is a subgroup of $\GL(2,\C)^3$ defined by
\begin{equation*}
 \Gamma:=\{ \gamma=I_2\times I_2\times M \bigm|
  \text{$\gamma$ commutes with $G$ in $\PGL(2,\C)^3$},~
  \det(M)=\pm 1 \}.
\end{equation*}
Since $\dim \Delta \leq \dim \Gamma$, $\phi$ is dominant if $\dim \Gamma \leq \dim V-\dim W$.
This can be checked directly, as indicated in the following table.
\begin{equation*}
\begin{array}{|c|c|c|c|c|}
 \hline
 H & \Xi & \dim \Gamma & \dim V & \dim W \\ \hline
 C_1 & \emptyset          & 1 & 14 & 13 \\ \hline 
 C_2 & \{ (1/4,1/4) \}    & 1 & 8 & 7 \\ \hline 
 C_2 & \{ (1/4,0) \}      & 0 & 8 & 8 \\ \hline 
 C_3 & \{ (1/3,1/3) \}    & 0 & 5 & 5 \\ \hline 
 C_4 & \{ (1/8,1/8) \}    & 0 & 4 & 4 \\ \hline 
 C_4 & \{ (1/8,1/4) \}    & 0 & 4 & 4 \\ \hline 
 C_2 \times C_2 & \{ (1/4,0),(0,1/4) \}
  & 0 & 5 & 5 \\ \hline 
\end{array}
\end{equation*}
\end{proof}




\subsection{Uniqueness} \label{SECT_uniqueness}

In this section, we will prove the uniqueness theorem of Calabi--Yau actions (Theorem \ref{THM_uniqueness}). 
We will also show the non-existence of a K3 surface with a Calabi--Yau $G$-action for $G \cong (C_3 \times C_3) \rtimes C_2$ (Theorem \ref{non-existence C3^2}). 
Henceforth we fix a semi-direct product decomposition $G=H\rtimes \langle \iota \rangle$ of a Calabi--Yau group $G$ as in Proposition \ref{PROP_OS}.
The key to proving the uniqueness is the following lemma, whose proof will be given in Section \ref{Key Lemma Section}. 

\begin{Lem}[Key Lemma] \label{LEM_keylemma}
Let $S$ be a K3 surface with a Calabi--Yau $G$-action. 
Then there exists an element $v\in \mathrm{NS}(S)^G$ such that $v^2=4$.
\end{Lem}

First, we consider the (coarse) moduli space of K3 surfaces $S$ with a Calabi--Yau $G$-action.
Let $\Psi_G$ denote the set of actions $\psi\colon G\rightarrow \mathrm{O}(\Lambda)$ of $G$ on $\Lambda= U^{\oplus 3} \oplus E_{8} (-1)^{\oplus 2}$ 
such that there exist a K3 surface $S$ with a Calabi--Yau $G$-action and a $G$-equivariant isomorphism $H^2(S,\Z)\rightarrow \Lambda_\psi$.
Here we denote $\Lambda$ with a $G$-action $\psi$ by $\Lambda_\psi$.
The group $\mathrm{O}(\Lambda)$ acts on $\Psi_G$ by conjugation:
\begin{equation*}
 (\gamma\cdot \psi)(g)=\gamma \psi(g) \gamma^{-1},\quad
 \gamma\in \mathrm{O}(\Lambda),\ \psi\in\Psi_G, \ g\in G.
\end{equation*}
Define the period domain $\tilde{\mathcal{D}}^G$ by
\begin{equation*}
 \tilde{\mathcal{D}}^G
  :=\bigsqcup_{\psi\in \Psi'_G} \tilde{\mathcal{D}}^{G,\psi}, \quad
 \tilde{\mathcal{D}}^{G,\psi}:= \{
  \C\omega \in \mathbb{P}( (\Lambda_\psi)^H_\iota \otimes \C) \bigm|
  \langle \omega,\omega \rangle=0, \  \langle\omega,\overline{\omega} \rangle>0
 \},
\end{equation*}
 where $\Psi'_G$ is a complete representative system of the quotient $\Psi_\Lambda/\mathrm{O}(\Lambda)$.
For any K3 surface $S$ with a Calabi--Yau $G$-action, there exist a unique $\psi\in\Psi'_G$ and a $G$-equivariant isomorphism $\alpha\colon H^2(S,\Z)\rightarrow \Lambda_\psi$.
Under the period map, $S$ with the $G$-action corresponds to the period point
\begin{equation*}
 (\alpha\otimes\C)(H^{2,0}(S))\in \tilde{\mathcal{D}}^{G,\psi}
  \subset \tilde{\mathcal{D}}^G.
\end{equation*} 

\begin{Lem} \label{LEM_latticeaction}
Let $S$ be a K3 surface with a $G$-action.
If the induced action $\psi\colon G\rightarrow \mathrm{O}(\Lambda)$ (which is defined modulo the conjugate action of $\mathrm{O}(\Lambda)$) is an element in $\Psi_G$, the $G$-action on $S$ is a Calabi--Yau action.
\end{Lem}
\begin{proof}
In general, a symplectic automorphism $g$ of a K3 surface $S$ of finite order is characterized as an automorphism such that $H^2(S,\Z)_g$ is negative definite \cite[Theorem 3.1]{Ni}.
Also, an Enriques involution is characterized by Lemma \ref{thm_enriques_involution}.
\end{proof}

\begin{Prop} \label{PROP_moduli_S}
For the moduli space $\mathcal{M}^G$ of K3 surfaces $S$ with a Calabi--Yau $G$-action, the period map defined above induces an isomorphism
\begin{equation*}
 \tau\colon
 \mathcal{M}^G \rightarrow
 \bigsqcup_{\psi\in \Psi'_G} \left( {\mathcal{D}}^{G,\psi}
  /\mathrm{O}(\Lambda,\psi) \right).
\end{equation*}
Here
\begin{align*}
 \mathcal{D}^{G,\psi}:=& \{
  \C\omega\in \tilde{\mathcal{D}}^{G,\psi} \bigm|
  \langle \omega,\delta \rangle\neq 0 ~ (\forall \delta \in \Delta_\psi)
 \}, \\
 \Delta_\psi :=& \{ \delta\in (\Lambda_\psi)_G \bigm|
  \delta^2=-2 \}, \\
 \mathrm{O}(\Lambda,\psi):=& \{
  \gamma \in \mathrm{O}(\Lambda) \bigm|
  \gamma \psi(g)=\psi(g) \gamma ~ (\forall g \in G)
 \}. 
\end{align*}
\end{Prop}
\begin{proof}
Let $\C \omega=(\alpha\otimes\C)(H^{2,0}(S))\in \tilde{\mathcal{D}}^{G,\psi}$ be the period point of a K3 surface $S$ with a Calabi--Yau $G$-action.
We can check that $\C\omega$ modulo the action of $\mathrm{O}(\Lambda,\psi)$ is independent of the choice of $\alpha$.
Since $G$ is finite, there exists a $G$-invariant K\"ahler class
 $\kappa_S$ of $S$.
We have $\kappa:=(\alpha\otimes\R)(\kappa_S) \in (\Lambda_\psi)^G\otimes \R$.
For any $\delta\in\Delta_\psi$, we have $\langle \kappa,\delta \rangle= 0$, and thus $\langle \omega,\delta \rangle\neq 0$ (see Section \ref{SECT_k3}).
Therefore we see that $\C\omega \in \mathcal{D}^{G,\psi}$.
Assume that a K3 surface $S'$ with a Calabi--Yau $G$-action is mapped to the same point as $S$ by $\tau$.
Then there exists a $G$-equivariant isomorphism 
$\phi\colon H^2(S,\Z)\rightarrow H^2(S',\Z)$ such that $(\phi\otimes\C)(H^{2,0}(S))=H^{2,0}(S')$. 
By Lemma \ref{LEM_reflection}, we may assume that
 $(\phi\otimes\R)(\kappa_S)$ is a K\"ahler class of $S'$.
By Theorem \ref{Torelli}, $\phi$ induces a $G$-equivariant isomorphism between $S$ and $S'$.
Therefore, $\tau$ is injective.
Let $\C\omega_1 \in \mathcal{D}^{G,\psi}$.
By the definition of $\mathcal{D}^{G,\psi}$,
 for any $\delta\in\Lambda_\psi$ with $\delta^2=-2$ and
 $\langle \omega_1,\delta \rangle=0$,
 we have
 $\langle \delta,(\Lambda_\psi)^G \rangle \supsetneq \{ 0 \}$.
Hence there exists $\kappa_1\in (\Lambda_\psi)^G\otimes \R$
 such that $\omega_1$ and $\kappa_1$
 satisfy the conditions (\ref{cond_surj_1}) and (\ref{cond_surj_2})
 in Theorem \ref{SurjPed}.
Therefore, by Theorem \ref{SurjPed}, there exist a K3 surface $S_1$ and an isomorphism 
$\alpha_1\colon H^2(S_1,\Z)\rightarrow \Lambda_\psi$ such that 
$(\alpha_1^{-1} \otimes\C)(\C\omega_1)=H^{2,0}(S_1)$ and $(\alpha_1^{-1} \otimes\R)(\kappa_1)$ is a K\"ahler class of $S_1$.
By Theorem \ref{Torelli}, the $G$-action on $\Lambda_\psi$ induces a $G$-action on $S_1$ such that $\alpha_1$ is $G$-equivariant, which is a Calabi--Yau action by Lemma \ref{LEM_latticeaction}.
This implies the surjectivity of $\tau$.
%
\end{proof}

Next let us consider projective models of the K3 surfaces with a Calabi--Yau action.

\begin{Lem} \label{LEM_doublequadric}
Let $S$ be a K3 surface with a Calabi--Yau $G$-action.
Assume that there exists an element $v\in \mathrm{NS}(S)^G$ such that $v^2=4$.
Then there exists a $G$-invariant line bundle $\mathcal{L}$ on $S$ satisfying the following conditions.
\begin{enumerate}
\item \label{quad_first}
$\mathcal{L}^2=4$ and $h^0(\mathcal{L})=4$.
\item
The linear system $|\mathcal{L}|$ defined by $\mathcal{L}$ is base-point free and defines a map $\phi_{\mathcal{L}} \colon S\rightarrow \PP^3$.
\item
$\dim \phi_{\mathcal{L}}(S)=2$.
\item \label{quad_last}
The degree $\deg\phi_{\mathcal{L}}$ of the map
 $\phi_{\mathcal{L}} \colon S\rightarrow \phi_{\mathcal{L}}(S)$
 is $2$,
 and $\phi_{\mathcal{L}}(S)$ is isomorphic to either
 $\PP^1\times\PP^1$ or a cone (i.e.\ a nodal quadric surface).
\end{enumerate}
\end{Lem}
\begin{proof}
Note that the closure $\overline{\mathcal{K}}_S$ of $\mathcal{K}_S$ is the nef cone of $S$.
We may assume that $v$ is nef by Lemma \ref{LEM_reflection}.
Let $\mathcal{L}$ be a line bundle on $S$ representing $v$.
By \cite[Sections 4 and 8]{SD}, we have $h^0(\mathcal{L})=4$ and either of the following occurs.
\begin{itemize}
\item[(a)]
$\mathcal{L}$ is base-point free, $\dim \phi_{\mathcal{L}}(S)=2$, and $\deg\phi_{\mathcal{L}}=1$ or $2$.
Any connected component of $\phi_{\mathcal{L}}^{-1}(p)$ for any $p$ is either a point or an ADE-configuration.
\item[(b)]
$\mathcal{L}\cong\mathcal{O}_S(3E+\Gamma)$ and $|\mathcal{L}|=\{ D_1+D_2+D_3+\Gamma \bigm| D_i\sim E \}$, 
 where $E$ and $\Gamma\cong \PP^1$ are irreducible divisors such that $E^2=0$, $\Gamma^2=-2$ and $\langle E, \Gamma\rangle=1$.
\end{itemize}
In Case (b), the base locus $\Gamma\cong \PP^1$ of $|\mathcal{L}|$ is stable under the action of $\iota$ and thus $\iota$ has a fixed point in $\Gamma$, which is a contradiction.
Hence Case (a) occurs. 
Since the fixed locus of any (projective) involution of $\PP^3$ is at least $1$-dimensional, there exists a fixed point $p$ of the action of $\iota$ on $\phi_{\mathcal{L}}(S)$.
If $\deg \phi_{\mathcal{L}}=1$, then
 $S^{\iota}\neq\emptyset$ by Lemma \ref{LEM_fixedpoint}, which is a contradiction.
Hence $\deg\phi_{\mathcal{L}}=2$, and $\phi_{\mathcal{L}}(S)$ is an irreducible quadric surface in $\PP^3$, which is either $\PP^1\times \PP^1$ or a cone.
\end{proof}

\begin{Prop} \label{PROP_genericity}
For a generic point in $\mathcal{M}^G$, the corresponding K3 surface $S$ with a Calabi--Yau action $\chi\colon G\rightarrow \Aut(S)$ has a projective model as in Proposition \ref{PROP_triplet}.
More precisely, $S$ and $\chi$ are realized as the double covering of $\PP^1\times\PP^1$ branching along a $(4,4)$-curve and a projective $G$-action on $S$.
%
\end{Prop}
\begin{proof}
Let $S$ be a K3 surface with a Calabi--Yau $G$-action. 
By Lemmas \ref{LEM_keylemma} and \ref{LEM_doublequadric}, there exists a $G$-invariant line bundle $\mathcal{L}$ satisfying the conditions (\ref{quad_first})--(\ref{quad_last}) in Lemma $\ref{LEM_doublequadric}$.
Let $\phi_{\mathcal{L}}=u\circ \theta$
 be the Stein factorization of $\phi_{\mathcal{L}}$. 
Then $\theta(S)$ is a normal surface possibly with ADE-singularities, and $u$ is a finite map of degree $2$. 
Assume that $\phi_{\mathcal{L}}(S)$ is a cone with the singular point $p$.
By Lemma \ref{LEM_fixedpoint}, $u^{-1}(p)$ consists of two points $\overline{p}_1,\overline{p}_2$, which are interchanged by any $g\in G\setminus H$.
Hence each $\theta^{-1}(\overline{p}_i)$ is a $(-2)$-curve $C_i$ on $S$ and $C_1-C_2\in H^2(S,\Z)^H_\iota$ with $(C_1-C_2)^2=-4$.
In particular, the Picard number of $S$ is greater than the generic Picard number $=22-\operatorname{rank} H^2(S,\Z)^H_\iota$. 
Therefore, if the period point of $(S,\chi)$ is contained in
\begin{equation*}
 \bigsqcup_{\psi\in \Psi'_G}
  \{ \C\omega\in \mathcal{D}^{G,\psi} \bigm|
  \langle \omega,\delta \rangle\neq 0 ~
  (\forall \delta \in \Delta'_\psi)
  \}, \quad
 \Delta'_\psi := \{ \delta\in (\Lambda_\psi)^H_\iota \bigm|
  \delta^2=-4 \},
\end{equation*}
 then $\phi_{\mathcal{L}}(S)\cong \PP^1\times\PP^1$ and the branching curve $B$ of $u$ has at most ADE-singularities.
Moreover, since the canonical bundle of $S$ is trivial, the bidegree of $B$ is $(4,4)$.
%
%
\end{proof}

Let $S$ be a K3 surface with a Calabi--Yau action $\chi\colon G\rightarrow \Aut(S)$.
We will prove (Theorem \ref{THM_uniqueness})
 that the pair $(S,\chi(G))$ is unique up to equivariant deformation.
A pair $(S,\chi)$ represents a K3 surface $S$ with a Calabi--Yau $G$-action with a fixed group $G$, 
while a pair $(S,\chi(G))$ represents a K3 surface $S$ with a subgroup of $\Aut(S)$ which gives a Calabi--Yau $G$-action.
The difference is whether or not we keep track of a way to identify $G$ with a subgroup of $\Aut(S)$.
Let $\Aut_H(G)$ denote the subgroup of $\Aut(G)$ consisting of elements which preserve $H$.
The group $\Aut_H(G)$ acts on the moduli space $\mathcal{M}^G$ of pairs $(S,\chi)$ from the right by
\begin{equation*}
 (S,\chi)\cdot \sigma=(S,\chi\circ \sigma), \quad
 \sigma\in \Aut_H(G).
\end{equation*}
Note that $\Aut(G)$ does not necessarily act on $\mathcal{M}^G$ because the action of $H$ on $S$ is symplectic by definition.
The orbit of $(S,\chi)$ under the action of $\Aut_H(G)$ is identified with $(S,\chi(G))$, and $\mathcal{M}^G/\Aut_H(G)$ is considered as the moduli space of pairs $(S,\chi(G))$.

\begin{Thm} \label{THM_uniqueness}
Let $H=C_n \ (1\leq n\leq 6)$, $C_2 \times C_2$ or $C_2\times C_4$.
Then there exists a unique subgroup of $\mathrm{O}(\Lambda)$ induced by a Calabi--Yau $G$-action up to conjugation in $\mathrm{O}(\Lambda)$, that is,
\begin{equation} \label{eq_psi}
 \Psi_G = \mathrm{O}(\Lambda) \cdot \psi \cdot \Aut_H(G), \quad
 \psi\in\Psi_G.
\end{equation}
Moreover, we have
\begin{equation} \label{eq_dom}
 \mathcal{M}^G/\Aut_H(G) \cong
 {\mathcal{D}}^{G,\psi} / \Gamma_\psi, \quad
 \Gamma_\psi := \{ \gamma \in \mathrm{O}(\Lambda) \bigm|
  \gamma \psi(G) \gamma^{-1}=\psi(G) \}.
\end{equation}
In particular, the moduli space $\mathcal{M}^G/\Aut_H(G)$ of pairs $(S,\chi(G))$ as above is irreducible, and a pair $(S,\chi(G))$ exists uniquely up to equivariant deformation.
\end{Thm}
\begin{proof}
By Proposition \ref{PROP_genericity}, a generic pair $(S,\chi(G))$ has a projective model as in Proposition \ref{PROP_triplet}.
Hence the existence of Calabi--Yau $G$-actions follows from Proposition \ref{PROP_triplet}.
Also, the connectedness of $\mathcal{M}^G/\Aut_H(G)$ follows from Propositions \ref{PROP_triplet} and \ref{PROP_222section}.
The stabilizer subgroup $\Sigma$ of $\mathcal{D}^{G,\psi}/\mathrm{O}(\Lambda,\psi)$ in $\Aut_H(G)$ is given by
\begin{equation*}
 \Sigma= \{ \sigma\in \Aut_H(G) \bigm|
  \psi\circ\sigma=\gamma\cdot \psi ~
  (\exists \gamma\in \mathrm{O}(\Lambda)) \}.
\end{equation*}
Since $\Sigma$ is naturally isomorphic to $\Gamma_\psi/\mathrm{O}(\Lambda,\psi)$, we can check (\ref{eq_psi}) and (\ref{eq_dom}) by Proposition \ref{PROP_moduli_S}.
\end{proof}

\begin{Thm} \label{non-existence C3^2}
For $G \cong (C_3 \times C_3) \rtimes C_2$, there does not exists a K3 surface with a Calabi--Yau $G$-action, that is, $\mathcal{M}^G=\emptyset$. 
\end{Thm}
\begin{proof}
As in the proof of Theorem \ref{THM_uniqueness}, a generic pair $(S,\chi(G))$ admits a projective model as in Proposition \ref{PROP_triplet}. 
However, there is no such a projective model. 
\end{proof}

\subsection{Moduli Spaces of Complex Structures}
\label{CMS}
In this section, we describe the complex moduli spaces of Calabi--Yau threefolds of type K.
We in particular show the irreducibility of the moduli space with a prescribed Galois group $G$.  
Throughout this section, we fix a semi-direct product decomposition $G=H\rtimes \langle \iota \rangle$ as in Proposition \ref{PROP_OS}.
In Section \ref{SECT_uniqueness}, we studied the moduli space of K3 surfaces $S$ with a Calabi--Yau $G$-action, which is denoted by $\mathcal{M}_{S}^G$ instead of $\mathcal{M}^G$ in this section.

Let us consider the moduli problem of elliptic curves with a $G$-action given in Proposition \ref{PROP_OS}, which we also call a Calabi--Yau $G$-action. 
Let $\mathcal{M}_{E}^G$ denote the moduli space of elliptic curves with a Calabi--Yau $G$-action.
An element in $\mathcal{M}_{E}^G$ is the isomorphism class of an elliptic curve with a faithful translation action of $H$.
A faithful translation action of $C_2 \times C_2$ on an elliptic curve $E$ is given by a level $2$ structure on $E$. 
Therefore $\mathcal{M}_{E}^G$ for $H={C_2 \times C_2}$ is identified with the (non-compact) modular curve $Y(2):=\mathbb{H}/\Gamma(2)$.
In the same manner, $\mathcal{M}^G_{E}$ for $H={C_n}$ is identified with the modular curve $Y_1(n):=\mathbb{H}/\Gamma_1(n)$. 
For $H=C_2 \times C_4$, we want the moduli space of elliptic curves with  linearly independent $2$- and $4$-torsion points. 
It is not difficult to see that it is identified with the modular curve $Y(2 \mid 4):=\mathbb{H}/\Gamma(2 \mid 4)$, where 
$$
\Gamma(2 \mid 4):=\left\{\begin{bmatrix}
        a & b\\
        c & d\\
        \end{bmatrix} \in \SL(2,\Z) \Biggm| a-1 \equiv c \equiv 0 \bmod 2, \ b\equiv d-1\equiv 0 \bmod 4\right\}. 
$$ 
We summarize the argument above in the following lemma. 

\begin{Lem} \label{LEM_ellptic_moduli}
Let $1 \le n \le 6$. 
The moduli space $\mathcal{M}_{E}^{G}$ of elliptic curves with a Calabi--Yau $G$-action
 is irreducible and given by the following.
\begin{center}
 \begin{tabular}{|c|c|c|c|} \hline  
 $G$ & $C_2 \times C_2 \times C_2$ & $C_n \rtimes C_2$ & $C_2 \times D_8$\\ \hline 
 $\mathcal{M}_{E}^{G}$ &  $Y(2)$ &   $Y_1(n)$& $Y(2 \mid 4)$\\  \hline
 \end{tabular}
 \end{center}
\end{Lem}

\begin{Thm} \label{Moduli}
Let $\Aut_H(G)$ denote the subgroup of $\Aut(G)$ consisting of elements which preserve $H$.
The quotient space
$$
 (\mathcal{M}_{S}^{G} \times \mathcal{M}_{E}^{G})/\Aut_H(G)
$$
 is the coarse moduli space of Calabi--Yau threefolds of type K whose minimal splitting covering has the Galois group isomorphic to $G$. 
 The moduli space is in particular irreducible. 
\end{Thm}
\begin{proof}
Two Calabi--Yau threefolds $X$ and $Y$ of type K are isomorphic 
if and only if the corresponding minimal splitting coverings $S_X\times E_X$ and $S_Y\times E_Y$ are isomorphic as Galois coverings. 
Suppose that the Galois group is isomorphic to $G$. 
The condition is equivalent to the existence of an isomorphism $f\colon S_X\times E_X \to S_Y\times E_Y$ 
and an automorphism $\phi\in \Aut(G)$ such that the following diagram commutes:
\[\xymatrix{
G \ar[d] \ar[r]^{\phi} \ar@{}[dr]|\circlearrowleft & G\ar[d] \\ 
\Aut(S_X\times E_X) \ar[r]
 &  \Aut(S_Y\times E_Y), 
}\]
that is, $f (g\cdot x)=\phi(g)\cdot f(x)$ for any $g \in G$ and any $x\in S_X\times E_X$. 
Note that we have $\phi\in \Aut_H(G)$ because we fix a subgroup $H$ as in Proposition \ref{PROP_OS}.
Since a Calabi--Yau $G$-action on $S_X\times E_X$ induces that on each $S_X$ and $E_X$, 
$S_X\times E_X$ is represented by a point in $\mathcal{M}_{S}^{G} \times \mathcal{M}_{E}^{G}$. 
The quotient space $(\mathcal{M}_{S}^{G} \times \mathcal{M}_{E}^{G})/\Aut_H(G)$ is then the coarse moduli space of 
the isomorphism classes of Calabi--Yau threefolds of type K with Galois group isomorphic to $G$. 
The moduli space is irreducible because the action of $\Aut_H(G)$ on the set of connected components of $\mathcal{M}_{S}^{G}$ is transitive and $\mathcal{M}_{E}^{G}$ is irreducible by Theorem \ref{THM_uniqueness} and Lemma \ref{LEM_ellptic_moduli}. 
\end{proof}


Combining Proposition \ref{PROP_triplet}, Theorems \ref{THM_uniqueness}, \ref{non-existence C3^2} and \ref{Moduli}, 
we complete the proof of the main theorem (Theorem \ref{Main Thm}) of the present section.


\section{Key Lemma} \label{Key Lemma Section}

Let $S$ be a $K3$ surface with a Calabi--Yau $G$-action.
We fix a semi-direct product decomposition $G=H\rtimes \langle \iota \rangle$ as in Proposition \ref{PROP_OS}.
This section is devoted to the proof of the Key Lemma:

\begin{keylem}[Lemma \ref{LEM_keylemma}]
There exists an element $v\in \mathrm{NS}(S)^G$ such that $v^2=4$.
\end{keylem}


\subsection{Preparation}


\begin{Lem} \label{Rank}
Set $r:=\rk H^{2}(S,\Z)^{H}$.  
We then have 
$$ 
\rk H^{2}(S,\Z)^{G}=\frac{r}{2}-1, \ \ \ \rk H^{2}(S,\Z)^{H}_{\iota}=\frac{r}{2}+1. 
$$
\end{Lem}
\begin{proof}
Let $X$ denote a Calabi--Yau threefold $(S\times E)/G$ of type K. 
Since a holomorphic 2-form $\omega_{S}$ on $S$ is contained in $H^{2}(S,\C)_{\iota}^{H}$, 
we see that $H^{2}(S,\C)^{G}\subset H^{1,1}(S)$. 
We therefore have 
$$
H^{1,1}(X)
 \cong H^{1,1}(S\times E,\C)^{G}
 \cong ( H^{2}(S,\C)^{G} \otimes H^{0}(E,\C) )
  \oplus ( H^{0}(S,\C) \otimes H^{2}(E,\C) ).
$$
as $\C$-vector spaces and conclude that $h^{1,1}(X)=\rk H^{2}(S,\Z)^{G}+1$. 
On the other hand we have canonical isomorphisms of $\C$-vector spaces: 
\begin{equation*}
 H^{2,1}(X)\cong H^{2,1}(S\times E,\C)^{G}
 \cong (H^{1,1}(S)_{\iota}^{H}\otimes H^{1,0}(E))
  \oplus ( H^{2,0}(S) \otimes H^{0,1}(E) ).
\end{equation*}
By the decomposition
\begin{equation*}
 H^{2}(S,\C)^H_{\iota}\cong
 H^{2,0}(S) \oplus H^{1,1}(S)^H_{\iota} \oplus H^{0,2}(S),
\end{equation*}
 we conclude that $h^{2,1}(X)=\rk H^{2}(S,\Z)_{\iota}^{H}-1$. 
Since the Euler characteristic $e(X)=e(S \times E)/|G|=0$, we have $h^{1,1}(X)=h^{2,1}(X)$. 
Then the claim readily follows.  
\end{proof}

Recall that the action of $H$ on $S$ is symplectic.
Hence the quotient surface $S/H$ has at most ADE-singularities and the minimal resolution $\widetilde{S}$ of $S/H$ is again a $K3$ surface.
Let $\widetilde{\iota}$ denote the involution of $\widetilde{S}$ induced by $\iota$.

\begin{Lem} \label{LEM_iotatilde}
The involution of $S/H$ induced by $\iota$ has no fixed point.
In particular, $\widetilde{\iota}$ is an Enriques involution of $\widetilde{S}$.
\end{Lem}
\begin{proof}
If the involution of $S/H$ induced by $\iota$ has a fixed point, then the action of $h\iota$ on $S$ has a fixed point for some $h\in H$, which is a contradiction.
\end{proof}

Each irreducible curve $M_i$ which contracts under the resolution $\widetilde{S}\rightarrow S/H$ is a $(-2)$-curve.
We denote by $M$ the negative definite lattice generated by
 $\{M_i\}_i$ and set $K:=M^\bot_{H^2(\widetilde{S},\Z)}$.

\begin{Lem} \label{LEM_nonexistenceU2}
If $H=C_n ~ (3\leq n \leq 6)$,
 then there is no nef class $v\in K^{\widetilde{\iota}}$
 such that $v^2=4$.
\end{Lem}
\begin{proof}
We assume that a nef class $v\in K^{\widetilde{\iota}}$
 satisfies $v^2=4$ and derive a contradiction.
By Lemma \ref{LEM_iotatilde}, the (induced) action of $\iota$ on $S/H$ has no fixed point.
By the same argument as in the proof of Proposition \ref{LEM_doublequadric}, 
 the class $v$ induces a morphism
 $\widetilde{f}\colon \widetilde{S}\rightarrow \PP^3$ such that
 $\widetilde{f}(\widetilde{S})$ is a quadric surface
 and
 the degree of $\widetilde{f}$ is $2$.
Since we have $v\bot M$ by the assumption,
 the morphism $\widetilde{f}$ induces a morphism
 $f\colon S/H\rightarrow \PP^3$.
By the proof of Lemma \ref{PROP_genericity},
 we may assume that $f(S/H)\cong \PP^1\times \PP^1$
 by taking a generic $S$.
The action of $\iota$ on $S/H$ is of the form $\sigma\tau$,
 where $\sigma$ is induced by a symplectic involution of
 $\widetilde{S}$ and $\tau$ is the covering transformation of $f$.
Let $\overline{\tau}\in \Aut(S)$
 be a lift of $\tau$.
Note that $\overline{\tau}$ normalizes $H$.
Since $f$ induces a generically one-to-one morphism
 $S/\langle H,\overline{\tau} \rangle \rightarrow \PP^1\times \PP^1$,
 it follows that
 $S/\langle H,\overline{\tau} \rangle$ is smooth and
 that the action of $\tau$ fixes
 each singular point of $S/H$.
Hence
 the actions of a generator of $H$ and $\overline{\tau}$ are
 represented by the matrices
$
 \begin{bmatrix}
  \zeta_n & 0 \\ 
  0  & \zeta_n^{-1}
 \end{bmatrix}
$
 and
$
 \begin{bmatrix}
  0  & 1 \\ 
  1 & 0
 \end{bmatrix}
$
 respectively, in local coordinates around a point in $S^H$, where $\zeta_n:=\exp(2\pi i/n)$. 
Therefore we have $\overline{\tau}h\overline{\tau}=h^{-1} \text{ for any } h\in H$ ($\star$).

We checked that $\tau$ fixes each point in
 $\operatorname{Sing}(S/H)$.
Hence the action of $\sigma$ has $8$ fixed points
 $q_i\not \in \operatorname{Sing}(S/H)$ $(1\leq i \leq 8)$
 by Theorem \ref{Nik FixPoint}. 
Let $Q_i \subset S$ denote the inverse image of $q_i$,
 which consists of $|H|$ points. 
Take a point $p \in Q_i$. 
Since $H$ acts on $Q_i$ transitively, we can take a lift
 $\overline{\sigma}\in \Aut(S)$ of $\sigma$ such that
 $\overline{\sigma}\cdot p=p$.
The action of $\overline{\sigma}$ around $p$
 is locally identified with that of $\sigma$ around $q_i$. 
Therefore $\ord(\overline{\sigma})=2$. 
Since $\overline{\sigma}\,\overline{\tau}\in H\iota$, 
the condition ($\star$) implies that $\overline{\sigma}$ commutes with $H$.
Hence the action of $\overline{\sigma}$ on each $Q_i$
 is trivial or free.
If $n=3,5$ or $6$, this contradicts to the fact $|S^{\overline{\sigma}}|=8$.
Let us next consider the case $n=4$. 
Let $h \in H$ be a generator of $H$. 
By a similar argument, for each $Q_i$,
 we can check that the action of either
 $\overline{\sigma}$ or $\overline{\sigma}h^2$ on $Q_i$ is trivial. 
Therefore we have 
$\cup_{i=1}^8 Q_i=S^{\overline{\sigma}} \cup S^{\overline{\sigma}h^2}$.
On the other hand, Theorem \ref{Nik FixPoint} implies that 
 $|\cup_{i=1}^8 Q_i|=8 \cdot |H|=32$ and 
 $|S^{\overline{\sigma}}\cup S^{\overline{\sigma}h^2}| = 2 \cdot 8=16$.
This is a contradiction. 
\end{proof}


\subsection{Proof of the Key Lemma}
In the following, we write $L_{R}:=L\otimes_\Z R$ for a lattice $L$ and a $\Z$-module $R$. 
The bilinear form on $L$ naturally extends to that on $L_R$ which takes values in $R$. 
We denote by $\Z_p$ the $p$-adic integers.
Lattices over $\Z_p$,
 and their discriminant groups and forms are defined in a similar way
 to lattices (over $\Z$).
Note that a lattice over $\Z_2$ is not necessarily even.
Assume that $L$ is non-degenerate and even.
Then $A(L_{\Z_p})$ and $q(L_{\Z_p})$
 are the $p$-parts of $A(L)$ and $q(L)$ respectively
 (see \cite{Ni2} for details).
In particular, if $|\disc(L)|$ is a power of $p$,
 then we have
 $(A(L),q(L))\cong (A(L_{\Z_p}),q(L_{\Z_p}))$.

Some remarks are in order before the proof.
We fix an identification
 $H^2(S,\Z)=\Lambda:= U^{\oplus 3} \oplus E_{8} (-1)^{\oplus 2}$.
Since $H^{2,0}(S)$ is contained in $(\Lambda^H_\iota)_\C$,
 we have $\mathrm{NS}(S)^G=\Lambda^G$ by (\ref{EQ_ns_lat}).
By \cite{Ni2}, the $H$-invariant lattice $\Lambda^H$
 is non-degenerate,
 and the rank of $\Lambda^H$, which depends only on the group $H$,
 is given in the following table.
\begin{center}
 \begin{tabular}{|c|c|c|c|c|c|c|c|c|c|}  \hline
 $H$  & $C_{1}$ & $C_{2}$ & $C_{3}$ & $C_{4}$ & $C_{5}$ & $C_{6}$
  &  $C_2 \times C_2$ & $C_{2} \times C_{4}$
  & $C_{3} \times C_3$
  \\ \hline
 $\rank \Lambda^H$ & $22$ & $14$ & $10$  &  $8$  & $6$ &  $6$
  & $10$ & $6$  & $6$ \\ \hline
 \end{tabular}
\end{center}
Since $\Lambda^G(\HF)$ is contained in $\Lambda^\iota(\HF)$,
 which is isomorphic to $U\oplus E_8(-1)$
 by Theorem \ref{thm_enriques_involution},
 it follows that $\Lambda^G(\HF)$ is even.
Similarly, $K^{\widetilde{\iota}}(\HF)$
 is even by Lemma \ref{LEM_iotatilde}.
Since $G$ is finite,
 there exists a $G$-invariant K\"ahler class of $S$.
Therefore $\Lambda^G$ has signature $(1,\rank \Lambda^G-1)$.
Set
\begin{equation*}
 S':=S \setminus
  \{ p\in S \bigm| h\cdot p=p ~ (\exists h\in H, ~ h\neq 1) \},
\end{equation*}
 and let $\pi \colon S' \rightarrow \widetilde{S}$
 be the natural map. 
Since $S\setminus S'$ is a finite set, the pushforward $\pi_*$
 and Poincar\'e duality induce a natural map
\begin{equation*}
 f\colon \Lambda=H^2(S,\Z)\cong H_2(S,\Z)
 \rightarrow
 H_2(\widetilde{S},\Z)
  \cong H^2(\widetilde{S},\Z).
\end{equation*}
For any $x,y\in \Lambda^H$, we have $\LF{f(x),f(y)}=|H| \LF{x,y}$.
The map $f$ decomposes as
\begin{equation*}
 f\colon \Lambda \rightarrow (\Lambda^H)^\vee
  \rightarrow H^2(\widetilde{S},\Z),
\end{equation*}
 where the first map is the restriction of the first projection of
 the decomposition $\Lambda_\Q=(\Lambda^H)_\Q\oplus (\Lambda_H)_\Q$
 and the second map is the natural injection. 
Since $\Lambda^H/(\Lambda^G\oplus \Lambda^H_\iota)\cong (\Z/2\Z)^{\oplus l}$
 for some $l$ by Lemma \ref{Involution},
 we have $2(\Lambda^G)^\vee\subset (\Lambda^H)^\vee$.
Hence we find that 
\begin{equation*}
 f( 2(\Lambda^G)^\vee ) \subset
 f( (\Lambda^G)_\Q\cap (\Lambda^H)^\vee ) \subset
 K^{\widetilde{\iota}}.
\end{equation*}
Set $L:=\Lambda^G(\HF)$.
Then we have
 $2(\Lambda^G)^\vee \cong 2 L^\vee(\HF)\cong L^\vee(2)$.
Thus we have
\begin{equation*}
 L^\vee(|H|) \cong f( 2(\Lambda^G)^\vee )(\HF)
 \subset K^{\widetilde{\iota}}(\HF).
\end{equation*}
Therefore $L$ satisfies the following conditions.
\begin{enumerate}
\item \label{COND_evenness}
 $L$ and $L^\vee(|H|)$ are even.
\item \label{COND_nonrep}
 If $H=C_n ~ (3\leq n \leq 6)$,
 then $v^2\neq 2/n$ for any $v \in L^\vee$.
\end{enumerate}
Here (\ref{COND_nonrep}) is a conclusion of
 Lemmas \ref{LEM_reflection} and \ref{LEM_nonexistenceU2}.
These conditions are derived from geometry of K3 surfaces.
On the other hand,
 the argument below is essentially lattice theoretic.

\begin{proof}[Proof of Key Lemma]
If $\Lambda^G$ contains $U(2)$,
 we see that the assertion of the Key Lemma holds.

 \underline{Case $H=C_1$.}~
We have $\Lambda^G\cong U(2)\oplus E_8(-2)$ by Theorem \ref{thm_enriques_involution}.

\underline{Case $H=C_2$.}~
This case has been studied by Ito and Ohashi (No.~13 in their paper \cite{IO}). 
They showed that $\Lambda^G\cong U(2)\oplus D_4(-2)$.
Here we give a proof of this fact for the sake of completeness.
By Lemma \ref{Rank},
 the signature of $\Lambda^G$ is $(1,5)$.
For each prime $p$,
 the lattice $L_{\Z_p}$ over the local ring $\Z_p$
 admits an orthogonal decomposition
 $L_{\Z_p} \cong \bigoplus_{i\geq 0} L^{(p)}_i(p^i)$,
 where $L^{(p)}_i$'s are unimodular lattices
 (see \cite{Ni2} for details).
By (\ref{COND_evenness}), we have $L^{(p)}_i=0$
 for any $p\neq 2$ and any $i\geq 1$.
Thus $|\disc(L)|$ is a power of $2$.
Again, by (\ref{COND_evenness}),
 $L^{(2)}_0$ and $L^{(2)}_1$
 are even, and we have $L^{(2)}_i=0$ for any $i\geq 2$.
Let $V$ denote the lattice over $\Z_2$ defined by the matrix
 $\QMat{2}{1}{1}{2}$.
In general, a lattice over $\Z_2$ is expressed
 as an orthogonal sum of the lattices in the following table
 \cite[Propositions 1.8.1 and 1.11.2]{Ni2}.
\begin{equation*}
 \begin{array}{|c|c|c|c|} \hline
 N & \LF{2^k a} & U(2^k)=\QMat{0}{2^k}{2^k}{0}
  & V(2^k)=\QMat{2^{k+1}}{2^k}{2^k}{2^{k+1}} \\  \hline
 A(N) & \Z/2^k \Z & (\Z/2^k \Z)^{\oplus 2}
  & (\Z/2^k \Z)^{\oplus 2} \\ \hline
 q(N) & \LF{a/2^k} & u(2^k):=\QMat{0}{1/2^k}{1/2^k}{0}
  & v(2^k):=\QMat{1/2^{k-1}}{1/2^k}{1/2^k}{1/2^{k-1}} \\ \hline
 \sign q(N) & a+k (a^2-1)/2
     & 0 & 4k  \\  \hline
 \end{array}
\end{equation*}
Here $k\geq 0$ and $a=\pm 1, \pm 3$.
(Note that $\LF{\pm 1/2}\cong \LF{\mp 3/2}$.)
Since $L^{(2)}_0$ and $L^{(2)}_1$ are even,
 $L_{\Z_2}$ has an orthogonal decomposition
\begin{equation*}
 L_{\Z_2} \cong U^{\oplus \nu} \oplus V^{\oplus \mu}
  \oplus U(2)^{\oplus \nu'} \oplus V(2)^{\oplus \mu'}.
\end{equation*}
Then we have
\begin{equation*}
 A(L)\cong (\Z/2\Z)^{\oplus 2(\nu'+\mu')}, \quad
 \sign q(L) \equiv 4 \mu' \bmod 8.
\end{equation*}
%
Since we have $\Lambda^\iota(\HF) \cong U\oplus E_8(-1)$ and
 $\Lambda^G=(\Lambda^\iota)^H$,
 it follows that $\nu'+\mu' \leq 2$
 by Proposition \ref{Nik Disc form} and Lemma \ref{Involution}.
The fact that $\sign \Lambda^G \equiv \sign q(L) \bmod 8$
 implies that $\mu'=1$.
Hence either of the following cases occurs.
\begin{equation*}
 \begin{array}{|c|c|c|c|} \hline
   & A(L)
   & q(L)
   & L \\
 \hline
 (a) & (\Z/2\Z)^{\oplus 2} & v(2)
  & U\oplus D_4(-1) \\  \hline
 (b) & (\Z/2\Z)^{\oplus 4}
     & u(2)\oplus v(2) 
  & U(2) \oplus D_4(-1)  \\  \hline
 \end{array}
\end{equation*}
Here, in each case,
 $L$ is uniquely determined by $q(L)$
 by Theorem \ref{Nik Genus}.
In Case (b), we have $q(\Lambda^\iota_H(\HF))\cong u(2)\oplus v(2)$
 by Proposition \ref{Nik Disc form}.
Since $q(\Lambda^\iota_H(\HF))$ takes values in $\Z/2\Z$,
 it follows that $\Lambda^\iota_H(1/4)$
 is an even unimodular lattice of rank $4$,
 which contradicts to the fact that
 any even unimodular lattice has rank divisible by $8$.
Hence Case (a) occurs: $\Lambda^G=L(2)\cong U(2)\oplus D_4(-2)$.

\underline{Case $H=C_3$.}~
The signature of $\Lambda^G$ is $(1,3)$.
By a similar argument, the condition (\ref{COND_evenness})
 implies that
 $A(L)\cong (\Z/3\Z)^{\oplus l}$
 for some $0\leq l \leq 4$.
Due to the relations
 $\LF{2/3}^{\oplus 2}\cong \LF{-2/3}^{\oplus 2}$ and
 $\sign \langle \pm 2/3 \rangle \equiv \pm 2 \bmod 8$
 (see \cite{Ni2} for details),
 we conclude that
\begin{equation*}
 q(L) \cong \LF{2/3}^{\oplus l-1} \oplus \LF{\pm 2/3},
 \quad
 \sign q(L) \equiv 2(l-1)\pm 2 \bmod 8.
\end{equation*}
We can check that either of the following cases occurs.
\begin{equation*}
 \begin{array}{|c|c|c|c|} \hline
   & A(L)
   & q(L)
   & L \\
 \hline
 (a) & \Z/3\Z & \langle -2/3 \rangle
  & U\oplus A_2(-1) \\  \hline
 (b) & (\Z/3\Z)^{\oplus 3}
     & \langle 2/3 \rangle^{\oplus 3} 
  & U(3) \oplus A_2(-1)  \\  \hline
 \end{array}
\end{equation*}
Case (b) cannot occur by (\ref{COND_nonrep}).
Therefore Case (a) occurs:
 $\Lambda^G\cong U(2)\oplus A_2(-2)$.

\underline{Case $H=C_4$.}~
The signature of $\Lambda^G$ is $(1,2)$.
Similarly, we find that $|\disc(L)|$ is a power of $2$.
Moreover, $L^{(2)}_0$ and $L^{(2)}_2$ are even,
 and we have $L^{(2)}_i=0$ for any $i\geq 3$.
If $L^{(2)}_0=L^{(2)}_2=0$,
 then $L(\HF) \cong U\oplus \LF{-1}$
 by the uniqueness of indefinite odd unimodular lattices.
Otherwise, $q(L) \cong \LF{-1/2}$ or $u(4)\oplus \LF{-1/2}$
 because of the relation
 $\LF{a/2^{k}} \oplus v(2^{k+1}) \cong \LF{5a/2^{k}} \oplus u(2^{k+1})$
 for any $a$ with $a \equiv 1 \bmod 2$ (see \cite{Ni2} for more details). 
Therefore we conclude that
 $L \cong U(2^k)\oplus \LF{-2}$ for $k=0,1$ or $2$.
By (\ref{COND_nonrep}),
 it follows that $k\neq 1,2$.
Thus $\Lambda^G\cong U(2)\oplus \LF{-4}$.

\underline{Case $H=C_5$.}~
In a similar way, we can check that
 $L$ is an indefinite even lattice of rank $2$
 such that
 $A(L)\cong (\Z/5\Z)^{\oplus l}$
 for some $0 \le l \le2$. 
By \cite[Table 15.2]{CS}, we see that 
\begin{equation*}
 L \cong U,\ \QMat{2}{1}{1}{-2} \ \text{or} \ U(5).
\end{equation*}
The second and third cases cannot occur by (\ref{COND_nonrep}).
Hence we conclude that $ \Lambda^G \cong U(2)$. 

\underline{Case $H=C_6$.}~
The signature of $\Lambda^G$ is $(1,1)$.
We make use of the argument in Case $H=C_3$.
Let $h$ be a generator of $H$.
We define
 $N:=\Lambda^{\langle h^2,\iota \rangle}(\HF)\cong U\oplus A_2(-1)$.
Then $h$ acts on $N$ as an involution and we have $N^h = L$.
By Lemma \ref{Involution},
 we can check that $A(N^h)$ and $A(N_h)$ are of the form
 $(\Z/2\Z)^{\oplus l}\oplus (\Z/3\Z)^{\oplus m}$
 for some $0 \le l\le 2$ and $0 \le m \le 1$.
Therefore, according to \cite[Tables 15.1 and 15.2]{CS},
 we have
$$
 N^h\cong U, \ U(2) \text{ or }
 \pm\begin{bmatrix}2&0\\ 0&-6 \end{bmatrix}; \quad
 N_h\cong A_2(-1), \ 
 \begin{bmatrix}-2&0\\0&-2 \end{bmatrix}, \ 
 \begin{bmatrix}-2&0\\0&-6 \end{bmatrix} \text{ or } 
 A_2(-2).
$$
By (\ref{COND_evenness}), we have $N^h\cong U$ or $U(2)$.
Note that $N^h\oplus N_h$ is a sublattice of $N$ of finite index
 and that $N^h$ and $N_h$ are primitive sublattices of $N$.
Hence we have $N^h\cong U$ and $N_h\cong A_2(-1)$
 (see Proposition \ref{Nik Isotropic}).
Thus $\Lambda^G=N^h(2) \cong U(2)$.

\underline{Case $H=C_2 \times C_2$.}~
The signature of $\Lambda^G$ is $(1,3)$.
An almost identical argument to that in Case $H=C_4$ shows that
 $L \cong U(2^k) \oplus \LF{-2}^{\oplus 2}$ for $k=0,1$ or $2$.
In order to show $k\neq 2$,
 we make use of the argument in Case $H=C_2$.
Let $H'\cong C_2$ be a subgroup of $H$. 
Then we know that
 $N:=\Lambda^{\langle H',\iota \rangle}(\HF)\cong U\oplus D_4(-1)$.
Since $H/H'$ acts on $N$ as an involution and we have $N^{H/H'} = L$,
 it follows that $|\disc(L)|$ divides $2^4$
 by Lemma \ref{Involution}.
Therefore we have
 $\Lambda^G \cong U(2^{k+1}) \oplus \LF{-4}^{\oplus 2}$
 for $k=0$ or $1$, and the assertion of the Key Lemma holds.

\underline{Case $H=C_2 \times C_4$.}~
The signature of $\Lambda^G$ is $(1,1)$.
Let $H'\cong C_4$ be a subgroup of $H$.
We then have
 $N:=\Lambda^{\LF{H',\iota}}(\HF) \cong U\oplus \langle -2 \rangle$
 by the argument in Case $H=C_4$.
Similarly, we find that
 $|\disc(N^{H/H'})|$ and $|\disc(N_{H/H'})|$ divide $2^2$.
By \cite[Table 15.2]{CS}, we have
\begin{equation*}
 N^{H/H'} \cong U,\ U(2) \text{ or }
  \langle 2 \rangle \oplus \langle -2 \rangle; \quad
 N_{H/H'} \cong \LF{-2} \text{ or } \LF{-4}.
\end{equation*}
Therefore, by Proposition \ref{Nik Isotropic}, we have
 $N^{H/H'} \cong U$ or $\LF{2}\oplus \LF{-2}$,
 and $N_{H/H'} \cong \LF{-2}$.
Thus $\Lambda^G \cong U(2)$ or $\LF{4}\oplus \LF{-4}$.
Hence the assertion of the Key Lemma holds.

\underline{Case $H=C_3 \times C_3$.}~
The signature of $\Lambda^G$ is $(1,1)$.
We make use of the argument in Case $H=C_3$.
Let $H',H'' \cong C_3$ be subgroups of $H$
 such that $H=H' \times  H''$.
Then $N:=\Lambda^{\LF{H',\iota}}(\HF)\cong U\oplus A_2(-1)$.
By a similar argument to the proof of Lemma \ref{Involution},
 we have
 $\Lambda/(\Lambda^{H''} \oplus \Lambda_{H''}) \cong (\Z/3\Z)^{\oplus l}$
 for some $l$.
Hence $N/(L\oplus L^\bot_N)\cong (\Z/3\Z)^{\oplus m}$
 for some $0\leq m \leq 2$,
 and $|\disc(L)|$ and $|\disc(L^\bot_N)|$ divide $3^3$.
Therefore, by \cite[Tables 15.1 and 15.2]{CS}, we have 
$$
 L \cong  U, \ U(3) \ \text{or} \ 
 \pm \begin{bmatrix}2&3\\3&0\end{bmatrix}; \quad
 L^\bot_N \cong
 A_2(-1), \ A_2(-3) \ \text{or} \ \QMat{2}{1}{1}{14}.
$$
By Proposition \ref{Nik Isotropic},
 we can check that $L \cong U$ or $U(3)$,
 and that $L^\bot_N  \cong A_2(-1)$.
Assume that $L \cong U(3)$.
Note that $N':=L^\bot_N=\Lambda^{\LF{H',\iota}}_{H''}(\HF)$.
Hence, by interchanging $H'$ and $H''$, we have
\begin{equation*}
 N'':= \Lambda^{\LF{H'',\iota}}_{H'}(\HF)\cong A_2(-1), \quad
 L \oplus N' \oplus N''
 \subset \Lambda^\iota(1/2).
\end{equation*}
Since we have
 $N/(L\oplus N')\cong \Z/3\Z$,
 there exist elements $v\in L^\vee$ and $w\in (N')^\vee$ such that
 $v^2=2/3$, $(w')^2=-2/3$ and $v+w'\in N$.
Similarly, there exists an element $w'' \in (N'')^\vee$
 such that $(w'')^2=-2/3$ and
 $v+w''\in \Lambda^{\LF{H'',\iota}}(\HF)$.
This implies that $w'-w''\in \Lambda^\iota(1/2)$ and that
 $(w'-w'')^2=-4/3$.
Since $\Lambda^\iota(1/2)$ is integral,
 this is a contradiction.
%
Therefore we conclude that $\Lambda^G\cong U(2)$.
\end{proof}


\begin{Prop} \label{H(S)^G}
The $G$-invariant lattice $H^2(S,\Z)^G$
 is given by the following table%
\footnote{In the proof of Proposition \ref{PROP_triplet}, we already checked that Case $H=C_3 \times C_3$ does not occur.}.
\begin{center}
 \begin{tabular}{|c|c|c|} \hline 
 $G$  & $H^{2}(S,\Z)^{G}$ \\ \hline
 $C_{2}$ & $U(2)\oplus E_{8}(-2)$ \\ \hline
 $C_2 \times C_2$  &  $U(2)\oplus D_4(-2)$ \\ \hline
 $C_2 \times C_2 \times C_2$
  &  $U(2)\oplus \langle -4 \rangle^{\oplus 2} $ \\ \hline
 $D_{6}$ & $U(2)\oplus A_{2}(-2)$ \\ \hline
 $D_{8}$ & $U(2)\oplus \langle -4\rangle$ \\ \hline
 $D_{10}$ & $U(2)$ \\ \hline
 $D_{12}$ & $U(2)$ \\ \hline
 $C_{2}\times D_{8}$ & $U(2)$ \\  \hline
 \end{tabular}
 \end{center}
\end{Prop}
\begin{proof}
It is worth noting that $H^2(S,\Z)^G$ does not depend on the choice of $S$. 
By the proof of the Key Lemma,
 it suffices to show the assertion for
 $G=C_2 \times C_2$ and $C_2 \times D_8$.
Note that a generic K3 surface $S$ with a Calabi--Yau $G$-action is realized as a Horikawa model (Proposition \ref{PROP_triplet}) and thus $H^2(S,\Z)^G$ contains $U(2)$,
 which is the pullback of the N\'eron--Severi lattice of $\PP^1\times \PP^1$.
For $G=C_2 \times C_2$,
 we have $H^2(S,\Z)^G\cong U(2^{k+1}) \oplus \LF{-4}^{\oplus 2}$
 for $k=0$ or $1$ by the proof of the Key Lemma.
Hence we have $H^2(S,\Z)^G\cong U(2) \oplus \LF{-4}^{\oplus 2}$.
Similarly, for $G=C_2\times D_8$,
 we have $H^2(S,\Z)^G \cong U(2)$ or $\LF{4}\oplus \LF{-4}$.
We thus conclude that $H^2(S,\Z)^G \cong U(2)$.
%
\end{proof}


\section{Properties}
In this section, we will investigate some basic properties of Calabi--Yau threefolds of type K. 
The explicit description obtained in the preceding section plays a central role in our study. 
Throughout this section, $X$ is a Calabi--Yau threefold of type K and $\pi\colon S\times E \rightarrow X$ is the minimal splitting covering with Galois group $G$. 
We also fix a semi-direct decomposition $G=H \rtimes \langle \iota \rangle$. 

There exist $G$-equivariant Ricci-flat K\"ahler metrics $g_S$ and $g_E$ on $S$ and $E$ respectively \cite{Ya}. 
Then the product metric $g_S\times g_E$ on $S\times E$ descends to a Ricci-flat K\"ahler metric $g'$ and $g$ on the quotients $(S\times E)/H$ and $X$ respectively.  
Let $T:=S/\langle \iota \rangle$ be the Enriques surface with the metric $g_T$ induced by $g_S$. 
We denote by $\mathrm{Hol}_{h}(Y)$ the holonomy group of a manifold $Y$ with respect to a metric $h$ (we do not refer to a based point).  

\begin{Prop}
\begin{enumerate}
\item  $\mathrm{Hol}_{g_T}(T)\cong \{ A \in \UU(2) \ | \ \det A=\pm1\} \subset \UU(2)$.  
\item $\mathrm{Hol}_g(X) \cong \mathrm{S}(\UU(2)\times C_2)\subset \SU(3)$. 
\end{enumerate}
\end{Prop}
\begin{proof}
Since the holonomy group $\mathrm{Hol}_{g_T}(T)$ cannot be $\SU(2)$, it must be a $C_2$-extension of $\mathrm{Hol}_{g_S}(S)\cong \SU(2)$ in $\UU(2)$. 
Such an extension is unique and this proves the first assertion. 
In order to prove the second assertion, we first consider the quotient $(S\times E)/H$, which admits a smooth isotrivial K3 fibration $(S\times E)/H \rightarrow E/H$. 
Since the action of $H$ on $S$ is symplectic, we see that $\mathrm{Hol}_{g'}((S\times E)/H) \cong \SU(2)$. 
Therefore the holonomy group $\mathrm{Hol}_g(X)$ is an extension of $\SU(2)$ in $\SU(3)$ of index at most 2. 
Since $X$ contains an Enriques surface, we conclude that $\mathrm{Hol}_g(X)\cong \mathrm{S}(\UU(2)\times C_2) \subset \SU(3)$.  
\end{proof}

\begin{Prop}
The following hold.
\begin{enumerate}
\item $\pi_1(X)=(\Z\times \Z) \rtimes G$, where the $G$-action on $\Z\times \Z$ is identified with that on $\pi_1(E)$.  
\item  $H_1(X,\Z) \cong (\Z/2\Z)^n$, where the exponent $n$ is given by the following table. 
\begin{center}
 \begin{tabular}{|c|c|c|c|c|c|c|c|c|}  \hline
 $G$  & $C_{2}$ & $C_{2} \times C_2$ & $C_{2} \times C_2 \times C_2$ & $D_{6}$ & $D_{8}$ &  $D_{10}$  &  $D_{12}$
 & $C_{2}\times D_{8}$ \\ \hline
 $n$ & $3$ & $4$ & $5$ & $3$ & $4$ & $3$ & $4$ & $5$\\ \hline
 \end{tabular}
 \end{center}
 \end{enumerate}
 \end{Prop}
 \begin{proof}
 The first assertion readily follows from the exact  sequence $0 \rightarrow \pi_1(S\times E)\rightarrow \pi_1(X) \rightarrow G \rightarrow 0$, 
 and the Calabi--Yau $G$-action on $E$.
 The second follows from the fact that $H_1(X,\Z)\cong \pi_1(X)^{\mathrm{Ab}}$, 
 or the Cartan--Leray spectral sequence associated to the \'etale map $S\times E\rightarrow X$.  
\end{proof}

\begin{Prop}
There exists no isolated (smooth) rational curve on $X$. 
Here we say a curve is isolated if it is not a member of any non-trivial family. 
\end{Prop}
\begin{proof}
Suppose that there exists an isolated rational curve $C\subset X$. 
Since $\pi$ is \'etale, the pullback
 $\pi^{-1}(C)$ consists of $|G|$ isolated rational curves.  
On the other hand, there is no isolated rational curve on the product $S\times E$ as any morphism $\mathbb{P}^1\rightarrow E$ is constant and any smooth rational curve on any K3 surface has self-intersection number $-2$.
This leads us to a contradiction. 
\end{proof}

All rational curves show up in families (parametrized by the elliptic curve $E$).
It is shown that they do not contribute to Gromov--Witten invariants but the higher genus quantum corrections are present at least for the Enriques Calabi--Yau threefold \cite{MP}.

\begin{Prop}
$\mathrm{Aut}(X)=\mathrm{Bir}(X)$. 
\end{Prop}
\begin{proof}
Any birational morphism between minimal models is decomposed into finitely many flops up to automorphisms \cite{Ka}. 
Hence it is enough to prove that there exists no flop of $X$.  
In the case of threefolds, the exceptional locus of any flopping contraction must be a tree of isolated rational curves \cite[Theorems 1.3 and 3.7]{KM1}.
The previous proposition therefore shows that there exists no flop of $X$. 
\end{proof}

\begin{Prop} \label{Aut}
The following hold.
\begin{enumerate}
\item If $G \cong D_{10},\ D_{12} \text{ or } C_{2}\times D_{8}$, we have $|\Aut(X)|<\infty$ .  \label{Aut_fin}
\item If $G \cong C_2, \ C_2\times C_2, \ C_2 \times C_2 \times C_2, \ D_6 \text{ or } D_8$, and $X$ is generic in the moduli space, we have $|\Aut(X)|=\infty$. \label{Aut_inf}
\end{enumerate}
\end{Prop}
\begin{proof}
For $G\cong D_{10},\ D_{12} \text{ or } C_{2}\times D_{8}$, we have $\rho(X)=3$ and the intersection form on $H^2(X,\Z)$ splits into the product of three linear forms.  
Hence the assertion (1) follows from the result of \cite{LOP}. 

For $G \cong C_2, \ C_2\times C_2, \ C_2 \times C_2 \times C_2, \ D_6 \text{ or } D_8$, 
the K3 surface $S$ is realized as a $(2,2,2)$-hypersurface in $\PP^1\times \PP^1\times \PP^1$ (Proposition \ref{PROP_222section}). 
Hence, by \cite{CO}, $\Aut(S)$ contains $C_2*C_2*C_2$, which commutes with $G$. 
Therefore the assertion (2) follows. 

\end{proof}


\begin{Rem}
In Proposition \ref{Aut} (\ref{Aut_inf}), the genericity assumption is essential at least in the case $G=C_2$.
In fact, if $G=C_2$, it follows that $\Aut(X)$ is infinite if and only if the automorphism group of the Enriques surface $S/\LF{\iota}$ is infinite.
Although the automorphism group of a generic Enriques surface is infinite, there exist Enriques surfaces with finite automorphism group 
(see \cite{Ko} for the classification of such Enriques surfaces).  
\end{Rem}

It is known that the automorphism group of a Calabi--Yau threefold with $\rho=1,2$ is finite \cite{Og}. 
On the other hand, it is expected that there is a Calabi--Yau threefold with infinite automorphism group for each $\rho\ge4$ (see for example \cite{Bo,GM}).  
Proposition \ref{Aut} provides a supporting evidence for this folklore conjecture, giving examples for small and new $\rho$. 
It is an open problem whether or not a Calabi--Yau threefold with $\rho=3$ admits infinite automorphism group \cite{LOP}.

\section{Calabi--Yau Threefolds of Type A}
\label{typeA}
In this final section, we slightly change the topic and probe Calabi--Yau threefolds of type A.
Recall that a Calabi--Yau threefold is called of type A if it is an \'etale quotient of an abelian threefold. 
By refining Oguiso and Sakurai's fundamental work \cite{OS} on Calabi--Yau threefolds of type A,  
we will finally settle the full classification of Calabi--Yau threefolds with infinite fundamental group (Theorem \ref{Classification A and K}).

Let $A:=\C^d/\Lambda$ be a $d$-dimensional complex torus. 
There is a natural semi-direct decomposition $\Aut(A)=A\rtimes \Aut_{\mathrm{Lie}}(A)$, 
where the first factor is the translation group of $A$ and $\Aut_{\mathrm{Lie}}(A)$ consists of elements that fix the origin of $A$. 
We call the second factor of $g \in \Aut(A)$ the Lie part of $g$ and denote it by $g_0$. 
The fundamental result in the theory of Calabi--Yau threefolds of type A is the following. 

\begin{Thm}[{Oguiso--Sakurai \cite[Theorem 0.1]{OS}}] \label{type A}
Let $X$ be a Calabi--Yau threefold of type A. 
Then the following hold. 
\begin{enumerate}
\item $X=A/G$, where $A$ is an abelian threefold and $G$ is a finite group acting freely on $A$  
in such a way that either of the following is satisfied:
\begin{enumerate}
\item $G=\langle a \rangle \times \langle b \rangle \cong C_2 \times C_2$ and 
$$
a_0= \begin{bmatrix}
        1 & 0  &  0\\
        0  & -1 & 0\\
        0  &  0 &  -1\\
        \end{bmatrix}, \ \ \ 
b_0= \begin{bmatrix}
        -1 & 0  &  0\\
        0  & 1 & 0\\
        0  &  0 &  -1\\
        \end{bmatrix},
$$
\item $G=\langle a,b \ | \ a^4=b^2=abab=1 \rangle \cong D_8$ and 
$$
a_0= \begin{bmatrix}
        1 & 0  &  0\\
        0  &  0 & -1\\
        0  &  1 &  0\\
        \end{bmatrix}, \ \ \ 
b_0= \begin{bmatrix}
        -1 & 0  &  0\\
        0  & 1 & 0\\
        0  &  0 &  -1\\
        \end{bmatrix}, 
$$
\end{enumerate}
where $a_0$ and $b_0$ are the Lie part of $a$ and $b$ respectively and the matrix representation is the one given by an appropriate realization of $A$ as $\C^3/\Lambda$. 
\item In the first case, $\rho(X)=3$ and in the second case $\rho(X)=2$. 
 \item Both cases really occur.
\end{enumerate}
\end{Thm}

Theorem \ref{type A} provides a classification of the Lie part of the Galois groups of the minimal splitting coverings, 
where the Galois groups do not contain any translation element. 
We will see that, in contrast to Calabi--Yau threefolds of type K, Calabi--Yau threefolds of type A are not classified by the Galois groups of the minimal splitting coverings. 
That is, a choice of Galois group does not determine the deformation family of a Calabi--Yau threefold of type A. 
We improve Theorem \ref{type A} by allowing non-minimal splitting coverings as follows.

\begin{Prop} \label{Refinement}
Let $X$ be a Calabi--Yau threefold of type A. 
Then $X$ is isomorphic to the \'etale quotient $A/G$ of an abelian threefold $A$ by an action of a finite group $G$, where $A$ and $G$ are given by the following.
\begin{enumerate}
\item
$A=A'/T$, where $A'$ is the direct product of three elliptic curves $E_1,E_2$ and $E_3$: 
$$
 A':=E_1\times E_2\times E_3, \ \  E_i:=\C/(\Z\oplus \Z \tau_i), \ \ \tau_i \in \mathbb{H}
$$
and $T$ is one of the subgroups of $A'$ in the following table, which consists of $2$-torsion points of $A'$.
\begin{equation*}
\begin{array}{|c|c|c|c|}\hline
 T_1 & T_2 & T_3 & T_4 \\ \hline
 0
 & \langle (0,1/2,1/2)_{A'} \rangle
 & \langle (1/2,1/2,0)_{A'},(1/2,0,1/2)_{A'} \rangle
 & \langle (1/2,1/2,1/2)_{A'} \rangle \\ \hline
\end{array}
\end{equation*}
Here $(z_1,z_2,z_3)_{A'}$ denotes the image of $(z_1,z_2,z_3)\in\C^3$ in $A'$.
\item
$G\cong C_2 \times C_2$ or $D_8$.
\begin{enumerate}
\item 
If $G=\langle a \rangle \times \langle b \rangle \cong C_2 \times C_2$, then $G$ is generated by
\begin{align*}
 &a\colon (z_1,z_2,z_3)_A\mapsto (z_1+\tau_1/2,-z_2,-z_3)_A, \\
 &b\colon (z_1,z_2,z_3)_A\mapsto (-z_1,z_2+\tau_2/2,-z_3+\tau_3/2)_A.
\end{align*}
\item
If $G=\langle a,b \ | \ a^4=b^2=abab=1 \rangle \cong D_8$, then $\tau_2=\tau_3=:\tau$, $T=T_2$ or $T_3$, and
 $G$ is generated by
\begin{align*}
 a&\colon (z_1,z_2,z_3)_A\mapsto (z_1+\tau_1/4,-z_3,z_2)_A, \\
 b&\colon (z_1,z_2,z_3)_A\mapsto (-z_1,z_2+\tau/2,-z_3+(1+\tau)/2)_A.
\end{align*}
\end{enumerate}
\end{enumerate}
Moreover, each case really occurs. 
\end{Prop}

\begin{proof}
By Theorem \ref{type A}, $X$ is of the form $A/G$ with $G$ isomorphic to either $C_2 \times C_2$ or $D_8$. 
Let $\C^3/\Lambda$ be a realization of $A$ as a complex torus.
In the case $G\cong C_2 \times C_2$, we may assume that $G$ is generated by
\begin{align*}
 a&\colon (z_1,z_2,z_3)_A\mapsto (z_1+u_1,-z_2,-z_3)_A, \\
 b&\colon (z_1,z_2,z_3)_A\mapsto (-z_1,z_2+u_2,-z_3+u_3)_A,
\end{align*}
 after changing the origin of $A$ if necessary.
Hence $\Lambda$ is stable under the following actions:
\begin{align*}
 a_0&\colon (z_1,z_2,z_3)\mapsto (z_1,-z_2,-z_3), \\
 b_0&\colon (z_1,z_2,z_3)\mapsto (-z_1,z_2,-z_3).
\end{align*}
From this, we see that 
there exist lattices $\Lambda_i \subset \C$ for $i=1,2,3$ such that
\begin{equation*}
 2\Lambda \subset \Lambda_1\times \Lambda_2 \times \Lambda_3
  \subset \Lambda.
\end{equation*}
Let $e_1,e_2,e_3$ be the standard basis of $\C^3$.
Set 
$$
 \Lambda':=\Lambda'_1\times \Lambda'_2\times \Lambda'_3\subset \Lambda,
\ \ \ 
 \Lambda'_i:=\{ z\in\C \bigm| z e_i\in \Lambda \}.
$$
Then $\Lambda/\Lambda'$ is a $2$-elementary group, that is,
 $\Lambda/\Lambda' \cong (\Z/2\Z)^{\oplus n}$ for some $n$.
Since $a^2=b^2=(ab)^2=\id_A$ and the $G$-action is free, 
we have  $u_i\not\in \Lambda'_i$ but $2 u_i\in\Lambda'_i$.
Let $v=(v_1,v_2,v_3)\in\Lambda$.
Suppose $v_1\equiv u_1 \bmod \Lambda'_1$, then $(z_1,v_2/2,v_3/2) \in A^{a}$.
Hence we conclude that $v_1\not\equiv u_1 \bmod \Lambda'_1$.
Similarly, $v_i\not\equiv u_i \bmod \Lambda'_i$ for $i=2,3$.
Therefore, we may assume that there exist $\tau_i \in \mathbb{H}$ for $i=1,2,3$ such that 
\begin{enumerate}
\item
$\Lambda'_i=\Z\oplus \Z\tau_i$,
\item
$u_i \equiv \tau_i/2 \bmod \Lambda'_i$,
\item
$v_i \equiv 0 \text{ or } 1/2 \bmod \Lambda'_i$ for all $v\in\Lambda$,
\end{enumerate}
 after changing each coordinate $z_i$ if necessary.
Now that we can check the assertion of the theorem in this case by a direct computation.
In particular, $T=\Lambda/\Lambda'$ coincides with one in the table up to permutation of the coordinates.

Similarly, in the case $G\cong D_8$, we may assume that $G$ is generated by
\begin{align*}
 a&\colon (z_1,z_2,z_3)_A\mapsto (z_1+u_1,-z_3,z_2)_A, \\
 b&\colon (z_1,z_2,z_3)_A\mapsto (-z_1,z_2+u_2,-z_3+u_3)_A.
\end{align*}
We use the same notation as above.
It follows that $\Lambda'_2=\Lambda'_3$, $4u_1 \in \Lambda'_1$, $2u_1\not\in \Lambda'_1$,
 $2u_i \in \Lambda'_i,u_i \not \in\Lambda'_i$ for $i=2,3$.
We have
\begin{align*}
 ab &\colon  (z_1,z_2,z_3)_A \mapsto (-z_1+u_1,z_3-u_3,z_2+u_2)_A \\
 (ab)^2 &\colon  (z_1,z_2,z_3)_A \mapsto
  (z_1,z_2+u_2-u_3,z_3+u_2-u_3)_A.
\end{align*}
By $(ab)^2=1$ and $A^{ab}=\emptyset$, it follows that $(0,u_2-u_3,u_2-u_3)\in\Lambda$ and $u_2-u_3\not\in\Lambda'_2$.
Since the action of $SL(2,\Z)$ on the set of level $2$ structures on an elliptic curve is transitive,
we may assume that $\tau_2=\tau_3=:\tau$, $u_2=\tau_2/2$, $u_3=(1+\tau_3)/2$.
By a similar argument to the case $G\cong C_2 \times C_2$, we can check that $v_i \equiv 0$ or $1/2 \bmod \Lambda'_i$ for any $v=(v_1,v_2,v_3)\in\Lambda$.
In particular, we have $(0,1/2,1/2)\in\Lambda$.
Since $T=T_4$ implies that $(1/2,0,0)\in \Lambda'_1$, which is a contradiction, it follows that  $T$ is either $T_2$ or $T_3$.
Moreover, we can check that the action of $G$ has no fixed point for $T=T_2,T_3$.
\end{proof}

\begin{Rem}
The above four cases for $G\cong C_2 \times C_2$ have previously been studied by Donagi and Wendland \cite{DW}. 
\end{Rem}
As was mentioned earlier, in contrast to Calabi--Yau threefolds of type K, Calabi--Yau threefolds of type A are not classified by the Galois groups of the minimal splitting coverings. 
They are classified by the minimal totally splitting coverings, where abelian threefolds $A$ which cover $X$ split into the product of three elliptic curves. 

Together with Theorem \ref{Main Thm}, Proposition \ref{Refinement} finally completes the full classification of Calabi--Yau threefolds with infinite fundamental group:

\begin{Thm} \label{Classification A and K}
There exist precisely fourteen Calabi--Yau threefolds with infinite fundamental group, up to deformation equivalence. 
To be more precise, six of them are of type A and eight of them are of type K. 
\end{Thm}



\par\noindent{\scshape \small
Max Planck Institute for Mathematics\\
Vivatsgasse 7, 53111 Bonn, Germany}
\par\noindent{\ttfamily hashimoto@mpim-bonn.mpg.de}
\break
\par\noindent{\scshape \small
Department of Mathematics, Harvard University\\
1 Oxford Street, Cambridge MA 02138 USA }
\par \noindent{\ttfamily kanazawa@cmsa.fas.harvard.edu}


\begin{thebibliography}{99}
                 \bibitem{As}P. Aspinwall, An $N=2$ Dual Pair and a Phase Transition, Nucl.\ Phys.\ B 460 (1996), 57--76. 
                \bibitem{BHPV}W. P. Barth, K. Hulek, C. A. M. Peters and A. van de Ven, Compact complex surfaces, Second edition, Springer-Verlag, Berlin, 2004. 
                \bibitem{Be1}A. Beauville, Vari\'et\'es K\"ahleriennes dont la premi\`ere classe de Chern est nulle, J. Diff.\ Geom.\ 18 (1983), no.\ 4, 755--782.
                \bibitem{Be2}A. Beauville, Some remarks on K\"{a}hler manifolds with $c_{1}=0$, in Classification of Algebraic and Analytic Manifolds, K. Ueno, ed., Progress Math.\ 39 (1983), 1--26. 
                \bibitem{Bo}C. Borcea, On desingularized Horrock--Mumford quintics, J. Reine Angrew.\ Math.\ 421 (1991), 23--41.
                \bibitem{CO}S. Cantat and K. Oguiso, Birational automorphism groups and the movable cone theorem for Calabi-Yau manifolds of Wehler type via universal Coxeter groups, Amer.\ J. Math.\ 137 (2015), no.\ 4, 1013--1044.
                \bibitem{CS}J. H. Conway and N. J. Sloane, Sphere packings, lattices and groups, Third edition, Springer-Verlag, New York, 1999.
                \bibitem{DW}R. Donagi and K. Wendland, On orbifolds and fermion constructions,  J. Geom.\ Phys.\ 59 (2009), no.\ 7, 942--968.
                \bibitem{FHSV}S. Ferrara, J. A. Harvey, A. Strominger and C. Vafa, Second quantized mirror symmetry, Phys.\ Lett.\ B 361 (1995), 59--65. 
                \bibitem{GM}A. Grassi and D. Morrison, Automorphisms and the K\"ahler cone of certain Calabi--Yau manifolds, Duke Math.\ J. 71 (1993), no. 3, 831--838.
                \bibitem{H}R. Hartshorne, Algebraic geometry, GTM 52, Springer-Verlag, New York-Heidelberg, 1977.
                \bibitem{HK}K. Hashimoto and A. Kanazawa, Calabi--Yau Threefolds of Type K (II) Mirror Symmetry, preprint.
                \bibitem{IO}H. Ito and H. Ohashi, Classification of involutions on Enriques surfaces, arXiv:1302.2316.
                \bibitem{Ka}Y. Kawamata, Flops connect minimal models, Publ.\ Res.\ Inst.\ Math.\ Sci.\ 44 (2008), 419--423.
                \bibitem{KM2}A. Klemm and M. Mari\~no, Counting BPS states on the Enriques Calabi--Yau, Comm.\ Math.\ Phys.\ 280 (2008), no.\ 1, 27--76.
                \bibitem{KM1}J. Koll\'ar and S. Mori, Birational geometry of algebraic varieties (English translation), Cambridge Tracts in Mathematics, 134, 1998.
                \bibitem{Ko}S. Kondo, Enriques surfaces with finite automorphism groups, Japan.\ J. Math.\ 12 (1986), 191--282.
                \bibitem{LOP}V. Lazi\'c, K. Oguiso and T. Peternell, Automorphisms of Calabi--Yau threefolds with Picard number three, arXiv:1310.8151.
                \bibitem{Ma}M. Mari\~no, Gromov--Witten invariants and topological strings: a progress report, International Congress of Mathematicians Vol.\ III, 409--419, Eur.\ Math.\ Soc., Z\"urich, 2006. 
                \bibitem{MP}D. Maulik and R. Pandharipande, New calculations in Gromov--Witten theory, PAMQ 4 (2008), 469--500. 
                \bibitem{Ni}V. V. Nikulin, Finite groups of automorphisms of K\"{a}hlerian surfaces of Type K3, Trudy Moskov.\ Mat.\ Obshch.\ 38 (1979), 75--137.
                \bibitem{Ni2}V. V. Nikulin, Integer symmetric bilinear forms and some of their geometric applications, Izv.\ Akad.\ Nauk SSSR Ser.\ Mat.\ 43 (1979), no. 1, 111--177. 
                \bibitem{Og}K. Oguiso, Automorphism groups of Calabi--Yau manifolds of Picard number two, to appear in J. Alg.\ Geom.
                \bibitem{OS} K. Oguiso and J. Sakurai, Calabi--Yau threefolds of quotient type, Asian J. Math.\ 5 (2001), no.1, 43--77. 
                \bibitem{Om} O. T. O'Meara, Introduction to quadratic forms, Springer-Verlag, Berlin, 1973.
                \bibitem{SD}B. Saint-Donat, Projective models of K3 surfaces, Math.\ Z, 189 (1985), 1083--1119. 
                \bibitem{Ya}S. -T. Yau, On the Ricci curvature of compact K\"{a}hler manifold and the complex Monge--Ampere equation I, Comm.\ Pure Appl.\ Math 31 (1978), no. 3, 339--411.
\end{thebibliography}
\end{document}